\documentclass[11pt]{article}

\usepackage{latexsym}
\usepackage{amssymb}
\usepackage{amsthm}
\usepackage{amscd}

\usepackage{amsmath}
\usepackage[small,nohug,heads=vee]{diagrams}

\usepackage{mathrsfs}
\usepackage[all]{xy}
\usepackage{hyperref} 
\usepackage[usenames,dvipsnames]{color}
\usepackage{graphicx,eepic}
\usepackage{float}

\usepackage{color}

\theoremstyle{definition}
\newtheorem* {theorem*}{Theorem}
\newtheorem* {proposition*}{Proposition}
\newtheorem* {conjecture*}{Conjecture}
\newtheorem{theorem}{Theorem}[section]

\theoremstyle{definition}
\newtheorem{problem}[theorem]{Problem}
\theoremstyle{definition}

\newtheorem* {example*}{Example}

\newtheorem{lemma}[theorem]{Lemma}
\theoremstyle{definition}
\newtheorem{definition}[theorem]{Definition}
\theoremstyle{definition}

\newtheorem{conjecture}[theorem]{Conjecture}
\newtheorem{proposition}[theorem]{Proposition}
\newtheorem{corollary}[theorem]{Corollary}

\newtheorem {remark}[theorem]{Remark}
\theoremstyle{definition}
\newtheorem {example}[theorem]{Example}
\theoremstyle{definition}

\theoremstyle{definition}

\theoremstyle{definition}

\xyoption{dvips}

\usepackage{fullpage}

\numberwithin{equation}{section}

\def\modu{\ (\mathrm{mod}\ }

\def\({\left(}
\def\){\right)}
   \newcommand{\FF}{\mathbb{F}}    \newcommand{\QQ}{\mathbb{Q}}   \newcommand{\cP}{\mathcal{P}} \newcommand{\cA}{\mathcal{A}}
 \newcommand{\cK}{\mathcal{K}} \newcommand{\cO}{\mathcal{O}} 
\newcommand{\cR}{\mathcal{R}}   
\newcommand{\cJ}{\mathcal{J}} 
\newcommand{\cC}{\mathcal{C}}

\def\NN{\mathbb{N}}
    \def\ZZ{\mathbb{Z}} \def\Aut{\mathrm{Aut}}  
         \def\spanning{\textnormal{-span}}   
  \def\wt{\widetilde}

\newcommand{\cM}{\mathcal{M}}
\newcommand{\cN}{\mathcal{N}}

\newcommand{\h}{\mathfrak{h}}

\def\barr{\begin{array}}
\def\earr{\end{array}}
\def\ba{\begin{aligned}}
\def\ea{\end{aligned}}
\def\be{\begin{equation}}
\def\ee{\end{equation}}

\def\cX{\mathcal{X}}

\def\qquand{\qquad\text{and}\qquad}
\def\quand{\quad\text{and}\quad}

\def\Des{\mathrm{Des}_L}
\def\omdef{\overset{\mathrm{def}}}

\def\hs{\hspace{0.5mm}}

\def\ds{\displaystyle}

\def\H{\mathcal H}
\def\cM{\mathcal M}
\def\id{\mathrm{id}}

\def\I{I^+}

\def\cR{\mathcal{R}_{\h}}

\def\QP{\text{QP}}

\def\ben{\begin{enumerate}}
\def\een{\end{enumerate}}

\def\h{\mathrm{ht}}
\def\hmin{\h_{\min{}}}

\makeatletter
\renewcommand{\@makefnmark}{\mbox{\textsuperscript{}}}
\makeatother

\allowdisplaybreaks[1]

\UseCrayolaColors

\begin{document}
\title{Bar operators for quasiparabolic conjugacy classes in a Coxeter group}
\author{Eric Marberg\footnote{This research was conducted with support from the National Science Foundation.}
\\
Department of Mathematics \\
Stanford University \\
{\tt emarberg@stanford.edu}}

\date{}

\maketitle

\begin{abstract}
The action of a Coxeter group $W$ on the set of left cosets  of a standard parabolic subgroup deforms to define a module $\mathcal{M}^J$ of the group's Iwahori-Hecke algebra $\mathcal{H}$ with a particularly simple form. Rains and Vazirani have introduced the notion of a \emph{quasiparabolic set} to characterize $W$-sets for which analogous deformations exist; a motivating example is the  conjugacy class of fixed-point-free involutions in the symmetric group. Deodhar has shown that the module $\mathcal{M}^J$ possesses a certain antilinear involution, called the bar operator, and a certain basis invariant under this involution, which generalizes the Kazhdan-Lusztig basis of $\mathcal{H}$.  The well-known significance  of this basis in representation theory makes it  natural to seek to extend Deodhar's results to the quasiparabolic setting. In general, the obstruction to finding such an extension is the existence of an appropriate  quasiparabolic analogue of the ``bar operator.''  In this paper, we consider the most natural definition of a quasiparabolic bar operator, and develop a theory of ``quasiparabolic Kazhdan-Lusztig bases'' under the hypothesis that such a bar operator exists. Giving content to this theory, we  prove that a bar operator in the desired sense does exist for quasiparabolic $W$-sets given by twisted conjugacy classes of twisted involutions. Finally, we prove several results classifying the quasiparabolic conjugacy classes in a Coxeter group.
\end{abstract}

\setcounter{tocdepth}{2}
\tableofcontents

\section{Introduction}

Let $(W,S)$ be a Coxeter system with length function $\ell : W \to \NN$, 
and let $\H=\H(W,S)$ be its {Iwahori-Hecke algebra}: this is  the $\ZZ[v,v^{-1}]$-algebra $\H$, with a basis given by the symbols $H_w$ for $w \in W$, whose  multiplication is uniquely determined by the condition that
\[ H_s H_w = \begin{cases} H_{sw} & \text{if }\ell(sw)>\ell(w) \\ H_{sw} + (v-v^{-1})\cdot H_w &\text{if $\ell(sw)<\ell(w)$}
\end{cases}
\qquad\text{for $s \in S$ and $w \in W$.}
\]
Observe that $H_1$ (which we typically write as 1 or omit) is the multiplicative unit of $\H$ and that $H_s$ is invertible for each $s \in S$. There exists a unique ring homomorphism $\H\to \H$ with $v\mapsto v^{-1}$ and $ H_s \mapsto H_s^{-1}$; we denote this map by $H\mapsto \overline{H}$, and refer to it as the \emph{bar operator} of $\H$.

Certain representations of $W$ admit natural and interesting deformations to modules of the algebra $\H$. For example, $\H$ viewed as a left module over itself clearly deforms the regular representation of $W$. 
For another example, suppose $J \subset S$ is a subset of simple generators and let $X=W/W_J $ be the set of left cosets of the standard parabolic subgroup $W_J = \langle J\rangle$ in $W$.
Define the height of a coset to be the minimal length of any of its elements, i.e.,  set
\[\label{height}\h(\cC) = \min_{w\in \cC} \ell(w)\qquad\text{for a left coset $\cC \in W/W_J$}.\]
Fix $u \in \{ -v^{-1},v\}$. For each choice of $u$, there is a unique $\H$-module structure on the free $\ZZ[v,v^{-1}]$-module generated by $W/W_J$ in which  $H_s \in \H$ for $s \in S$ acts  on cosets $\cC \in W/W_J$ by the formula
\be\label{action} H_s : \cC \mapsto
 \begin{cases} 
 s\cC &\text{if $\h(s\cC) > \h(\cC)$} \\
 s\cC + (v-v^{-1})\cdot \cC&\text{if $\h(s\cC) < \h(\cC)$} \\ 
 u\cdot \cC &\text{if $\h(s\cC) = \h(\cC)$.}
 \end{cases}
\ee
Denote these $\H$-modules by $\cM^J$ (when $u=v$) and $\cN^J$ (when $u=-v^{-1}$), respectively. 
Note that if we specialize the parameter $v$ to $1$, then $\cM^J$ and $\cN^J$ become the modules of the group ring $\ZZ W$ given by respectively inducing the trivial and sign representations of $W_J$ to $W$.

The formulas above  are well-defined if we replace $X=W/W_J$ by the set of cosets of any subgroup $H\subset W$. However, the assertion that \eqref{action} defines an $\H$-module structure only holds for some choices of $H$ and not for others, in a fashion which is not yet very well understood.
The following is therefore a natural question: given a $W$-set $X$ with a height function $\h : X \to \NN$, when does the free $\ZZ[v,v^{-1}]$-module generated by $ X$ have an $\H$-module structure described by the obvious analogue of \eqref{action}? 
Rains and Vazirani \cite{RV} identify a simple set of conditions which are sufficient for this phenomenon to occur, and call $W$-sets satisfying these conditions \emph{quasiparabolic}. 
We review the precise definition in Section \ref{qp-sect}; informally, a $W$-set is quasiparabolic if it has  a ``Bruhat order'' which is  compatible with its height function and which satisfies a few technical properties exactly analogous to the Bruhat order on $W$. 
The $W$-set of left cosets of any standard parabolic subgroup 
 is quasiparabolic.
More exotically, some but not all conjugacy classes of involutions in a Coxeter group 
are quasiparabolic, relative to the height function $\tfrac{1}{2}\ell$.

Let $ M \mapsto \overline{M}$ denote a $\ZZ$-linear map $ \cM^J \to \cM^J$. We call such a map a \emph{bar operator} of $\cM^J$ if it fixes 
the unique coset in $W/W_J$ of height zero and satisfies
\be\label{bar-eq} 
\overline{HM} = \overline{H}\cdot  \overline{M} \qquad\text{for all $H \in \H$ and $M \in \cM^J$.}
\ee
Define a \emph{bar operator} of $\cN^J$ analogously.
In \cite{Deodhar,Deodhar2}, Deodhar shows that 
$\cM^J$ and $\cN^J$ both admit unique bar operators,
and proves 
that each module has a unique basis of elements  invariant under the bar operator which is congruent to the ``standard basis'' of cosets $W/W_J$ modulo $v^{-1}\ZZ[v^{-1}]$-linear combinations of standard basis elements. 
These new bases are the \emph{parabolic Kazhdan-Lusztig bases} of $\cM^J$ and $\cN^J$; when $J=\varnothing$, they both may be identified with the well-known \emph{Kazhdan-Lusztig basis} of $\H$ introduced in \cite{KL}.

Rains and Vazirani show that the free $\ZZ[v,v^{-1}]$-module generated by a quasiparabolic set $X$ may be given two   $\H$-module structures, which we denote $\cM$ and $\cN$,
by a formula exactly analogous to \eqref{action}.
(We review the precise definitions in Section \ref{module-sect}.) One naturally asks whether there exists a notion of a ``quasiparabolic Kazhdan-Lusztig basis'' for these modules, which specializes to Deodhar's parabolic Kazhdan-Lusztig bases when $X = W/W_J$. 
The exploration of this question is the main topic of the present work.
As motivation, we recall that the (parabolic) Kazhdan-Lusztig bases attached to a Coxeter system display a number of remarkable properties not at all evident from their elementary definition, and have  connections to a surprising variety of topics in representation theory.
It seems reasonable to expect that  some  interesting properties and connections will likewise hold in the  quasiparabolic setting; \cite[\S9]{RV} presents several  phenomena along these lines.

The main obstruction to formulating a definition of a ``quasiparabolic Kazhdan-Lusztig basis'' is proving the existence a bar operator for the $\H$-modules $\cM$ and $\cN$.
For us, a \emph{bar operator} is a $\ZZ$-linear map $\cM \to \cM$ (respectively, $\cN \to \cN$) which fixes  elements of minimal height in each $W$-orbit in $X$ and   is compatible with the bar operator of $\H$ in the sense of \eqref{bar-eq}; see Definition \ref{barop-def}.
The following conjecture is equivalent to  
\cite[Conjecture 8.4]{RV} by \cite[Proposition 2.15]{RV}:

\begin{conjecture*}[Rains and Vazirani \cite{RV}]  If $X$ is a quasiparabolic set which is bounded below (in the sense that the heights of the elements in any given $W$-orbit are bounded below), then the corresponding modules $\cM$ and $\cN$ each have bar operators.
\end{conjecture*}

In this paper, we develop a number of general consequences of this conjecture, and also prove that the conjecture holds in some motivating cases of interest.
A more detailed 
 outline of our results goes as follows.
After stating some preliminaries in Section \ref{prelim-sect}, we devote Section \ref{s3-sect} to developing the general properties of bar operators, where we prove the following: 

\begin{theorem*}[See Section \ref{s3-sect}] Suppose $X$ is a quasiparabolic set which is bounded below.
%
If either of the corresponding modules $\cM$ or $\cN$ has a bar operator, then both modules have unique bar operators which determine each other and are involutions.
\end{theorem*}

We write that $X$ \emph{admits a bar operator} if both of the corresponding modules $\cM$ and $\cN$ do;
in this case, we prove that $\cM$ and $\cN$ each have a certain distinguished basis in the following sense:

\begin{theorem*}[See Theorem \ref{qpcanon-thm}] Assume $X$ is a quasiparabolic set which is bounded below and admits a bar operator. Then $\cM$ and $\cN$ each  have a unique ``canonical basis,'' by which we mean a  basis of elements invariant under the corresponding bar operator which is congruent to the ``standard basis''  $X$ modulo $v^{-1}\ZZ[v^{-1}]$-linear combinations of standard basis elements.
\end{theorem*}

These bases generalize Deodhar's parabolic Kazhdan-Lusztig bases, and in
 Section \ref{cb-sect} we show that they retain many of the same properties.
In Section \ref{Wgraph-sect} we prove that the canonical bases of $\cM$ and $\cN$ define two ways of viewing the quasiparabolic set $X$ as a $W$-graph. 

For the preceding theorems to be of  interest we must have other examples of quasiparabolic sets with bar operators, besides the motivating case  $X=W/W_J$. 
In Section \ref{qcc-sect} 
we describe a source of such 
quasiparabolic sets.
Let $\theta : W \to W$ be a group automorphism with $\theta(S) = S$. Then $W$ acts on itself by the twisted conjugation $w : x \mapsto w\cdot x\cdot \theta(w)^{-1}$; an orbit under this action is a \emph{twisted conjugacy class}; and an element $x \in W$ is a \emph{twisted involution} (relative to $\theta$) if $x^{-1} = \theta(x)$.

\begin{theorem*}[See Theorem \ref{Ibarop-thm}] Any twisted conjugacy class of twisted involutions (relative to $\theta$) which is quasiparabolic (relative to the height function $\frac{1}{2}\ell$) admits a bar operator.
\end{theorem*}

This result applies, in particular, to Rains and Vazirani's motivating example of the conjugacy class of fixed-point-free involutions in the symmetric group, which thus index two  ``quasiparabolic Kazhdan-Lusztig bases.'' 
In Sections \ref{cc-sect} and \ref{suff-sect} we prove several results which control which twisted conjugacy classes are quasiparabolic. 
Among these are the following statements, which show that the previous theorem's restriction to the case of twisted involutions is   not so limiting:

\begin{theorem*}[See Corollary \ref{big-cor3}] In an arbitrary Coxeter group, all  (ordinary) conjugacy classes  which are quasiparabolic (relative to the height function $\frac{1}{2}\ell$) consist of  involutions.
\end{theorem*}

\begin{theorem*}[See Theorem \ref{qp-finite-thm}] In a finite Coxeter group, all  twisted conjugacy classes  which are quasiparabolic (relative to the height function $\frac{1}{2}\ell$) consist of twisted involutions.
\end{theorem*}


There can exist quasiparabolic twisted conjugacy classes which do not consist of twisted involutions; we construct examples in a necessarily infinite Coxeter group in Section \ref{suff-sect}.
%
In the last section of the paper, we list a number of open questions and problems.

\subsection*{Acknowledgements}

I thank Daniel Bump, Michael Chmutov, John Stembridge, and Zhiwei Yun for  helpful discussions.

\section{Preliminaries}\label{prelim-sect}

In this section 
 $(W,S)$ denotes an arbitrary Coxeter system with length function  $\ell $. We write $\leq$ for the Bruhat order on $W$. Recall that if $s \in S$ and $w \in W$ then $sw<w$ if and only if $\ell(sw) = \ell(w)-1$.

\subsection{Quasiparabolic sets}\label{qp-sect}

Rains and Vazirani introduce the following definitions in \cite[\S2]{RV}.

\begin{definition}\label{scaled-def}
A \emph{scaled $W$-set} is a $W$-set $X$ with a height function $\h : X \to \QQ$ satisfying
\[|\h(x) - \h(sx)| \in \{0,1\}\qquad\text{for all $s \in S$ and $x \in X$.}\]
\end{definition}


Denote the set of reflections in $W$ by
$R = \{ wsw^{-1} : w \in W \text{ and }s \in S\}.$
 
\begin{definition}\label{qp-def} A scaled $W$-set $(X,\h)$ is \emph{quasiparabolic} if both of 
the following properties hold:
\ben
\item[] \hspace{-7mm}(QP1) If $\h(rx) = \h(x)$ for some $(r,x) \in R\times X$ then $rx =x$.

\item[] \hspace{-7mm}(QP2) If  $\h(rx) > \h(x)$ and $\h(srx) < \h(sx)$ for some $(r,x,s) \in R\times X \times S$ then $rx=sx$.

\een
\end{definition}

\begin{example}\label{case0-ex}
The set  $W$ with height function $\h=\ell$
is quasiparabolic relative to its action on itself by left (also, by right) multiplication
and also when viewed as a scaled $W\times W$-set relative to the action $(x,y) : w\mapsto xwy^{-1}$;
 see \cite[Theorem 3.1]{RV}.
\end{example}

\begin{example}\label{parabolic-ex}
Let $J \subset S$ and define  $W^J = \{ w\in W  : ws > w \text{ for all }s \in J\}$. It is well-known that any element $w \in W$ has a unique factorization $w = uv$ with $u \in W^J$ and $v \in W_J=\langle J\rangle$.
Define 
\[ s \bullet w =\begin{cases} sw &\text{if $sw \in W^J$} \\ w &\text{otherwise}\end{cases}
\qquad \text{for $s \in S$ and $w \in W^J$.}\]
Then $\bullet : S\times W^J \to W^J$ extends to an action of $W$ on $W^J$, which is isomorphic to the natural action of $W$ on $W/W_J$. 
The $W$-set $W^J$ is quasiparabolic relative to the height function $\h = \ell$. 
This example is fundamental, and motivates the name  ``quasiparabolic.''
\end{example}

\begin{example}\label{cc-ex}
A conjugacy class in $W$ is a scaled $W$-set relative to  conjugation  and the  height function $\h=\ell/2$. This scaled $W$-set is sometimes but not always quasiparabolic; see Section \ref{cc-sect}.
\end{example}

We restate \cite[Corollary 2.13]{RV} as the  lemma which follow this definition:

\begin{definition} 
An element $x $ in a scaled $W$-set  $X$ is \emph{$W$-minimal} (respectively, \emph{$W$-maximal}) if $\h(sx) \geq \h(x)$ (respectively, $\h(sx) \leq \h(x)$) for all $s \in S$.
\end{definition}

\begin{lemma}[Rains and Vazirani \cite{RV}] \label{minimal-lem} 
If a scaled $W$-set is quasiparabolic, then each of its orbits contains at most one $W$-minimal element and at most one $W$-maximal element. These elements, if they exist, have  minimal (respectively, maximal) height in their $W$-orbits.
\end{lemma}

\begin{remark}
This property is enough   to nearly classify the quasiparabolic conjugacy classes in the symmetric group. Assume that $W = S_n$ and $S = \{ s_i = (i,i+1) : i =1,\dots,n-1\}$.
 Suppose $\cK \subset S_n$ is a quasiparabolic conjugacy class (relative to the height function $\h = \ell/2$).
 Since $\cK$ is finite, it contains a unique $W$-minimal element by Lemma \ref{minimal-lem}. 
 As every permutation is conjugate in $S_n$ to its inverse (which has the same length), $\cK$ must consists of involutions. 
 There are $1+\lfloor n/2 \rfloor$ such conjugacy classes: $\{1\}$ and the conjugacy classes of $s_1s_3s_5\cdots s_{2k-1}$ for positive integers $k$ with $2k\leq n$. While $\{1\}$ is trivially quasiparabolic,  the conjugacy class of $s_1s_3s_5\cdots s_{2k-1}$ is quasiparabolic only if $2k =n$, since otherwise  $s_2s_4s_6\cdots s_{2k}$ belongs to the same conjugacy class but has the same (minimal) length.
The only remaining conjugacy class, consisting of the fixed-point-free involutions in $S_n$ for $n$ even, is  quasiparabolic by \cite[Theorem 4.6]{RV}. 
\end{remark}

For the rest of this section, $(X,\h)$ denotes  a fixed quasiparabolic $W$-set.
The following lemma is a consequence of \cite[Theorem 2.8]{RV}.
\begin{lemma}[Rains and Vazirani \cite{RV}] 
\label{exchange-lem}
 Suppose $x_0 \in X$ is a $W$-minimal element.
%
The set 
\be\label{cR(x)}\cR(x) \omdef= \{ w \in W : x=wx_0\text{ such that }\h(x) = \ell(w) + \h(x_0)\}\ee
is then nonempty for any element $x$ in the $W$-orbit of $x_0$.

\end{lemma}

Additionally, we have this definition from \cite[\S5]{RV}, which attaches to $X$ a certain partial order:

 \begin{definition}\label{bruhat-def}
 The \emph{Bruhat order} on 
 a quasiparabolic $W$-set 
 $X$ is  the weakest partial order $\leq$  with
$ x \leq rx $ for all $x \in X $ and $r\in R$ with $ \h(x) \leq \h(rx)$.
\end{definition}

\begin{remark}
If $(X,\h)$ is one of the quasiparabolic $W$-sets in Examples \ref{case0-ex} or \ref{parabolic-ex}, 
then the Bruhat order coincides with the usual Bruhat order on $W$ restricted to $X$.
If $X$ is a quasiparabolic conjugacy class in $W$ as in Example \ref{cc-ex}, then the Bruhat order on $W$ restricts to an order which is equal to or stronger than the Bruhat order on $X$ (viewed as a quasiparabolic set). In all known examples these two orders actually coincide, but showing whether this holds in general is an open problem; see the remarks following \cite[Proposition 5.17]{RV} and also Conjecture \ref{bruhat-conj}.
If these two orders were always equal, it would follow from \cite[Proposition 5.16]{RV} that any quasiparabolic conjugacy class is a graded poset with respect to the order induced by the usual Bruhat order, a property which does not hold for arbitrary conjugacy classes in Coxeter groups.
\end{remark}

It follows immediately from the definition that if $x,y \in X$ then $x < y$ implies $\h(x) < \h(y)$.  Rains and Vazirani develop in 
\cite[Section 5]{RV} several other general properties of the Bruhat order. Among other facts,
they show that the set $X$ is a graded poset relative to $\leq$, and that the length of every maximal chain in the Bruhat order between $x\leq y$
is $\h(y) - \h(x)$ \cite[Proposition 5.16]{RV}. 
We note explicitly the following  lemma (which appears as \cite[Lemma 5.7]{RV}) for use later:
 
\begin{lemma}[Rains and Vazirani \cite{RV}]\label{bruhat-lem}
Let $x,y \in X$  such that $x\leq y$ and $s \in S$. Then
 \[sy \leq y \ \Rightarrow\ sx \leq y
 \qquand
 x\leq sx
 \ \Rightarrow\ x \leq sy.
 \]
\end{lemma}


\subsection{Hecke algebra modules} \label{module-sect}

Let $\cA = \ZZ[v,v^{-1}]$ 
and recall that the \emph{Iwahori-Hecke algebra} of  $(W,S)$
is the $\cA$-algebra 
 \[\H = \H(W,S)=\cA\spanning\{ H_w : w \in W\}\]
defined in the introduction.
For background on this algebra, see, for example, \cite{CCG,Hu,KL,Lu}.
Observe that $H_s^{-1} = H_s + (v^{-1}-v)$ and that $H_w = H_{s_1}\cdots H_{s_k}$ whenever $w=s_1\cdots s_k$ is a reduced expression. 
Hence  every basis element $H_w$ for $w \in W$ is invertible.

Rains and Vazirani show that the permutation representation of $W$ on a quasiparabolic set deforms to a well-behaved representation of $\H$.
In detail, for any scaled $W$-set 
$(X,\h)$ 
 let 
\[ \cM=\cM(X,\h)= \cA\spanning\{ M_x : x \in X\} \quand \cN = \cN(X,\h)
=
 \cA\spanning\{ N_x : x \in X\}\]
denote the free $\cA$-modules with  bases given by the symbols $M_x$ and $N_x$ for $x \in X$. We call $\{ M_x \}_{x \in X}$ and $\{ N_x \}_{x \in X}$  the \emph{standard bases} of $\cM$ and $\cN$, respectively.
We view the $\cA$-modules $\cM$ and $\cN$ as distinct $\H$-modules according to the following result, which appears as \cite[Theorem 7.1]{RV}.

\begin{theorem}[Rains and Vazirani \cite{RV}] \label{module-thm}
Assume    $(X,\h)$ is a quasiparabolic $W$-set.
\ben
\item[(a)] 
There is a  unique $\H$-module structure on $\cM$ such that for all  $s \in S$ and $x \in X$
\[ H_s M_x 
= 
\begin{cases}
M_{sx}&  							 \text{if $\h(sx)>\h(x)$} \\ 
M_{sx}  +  (v-v^{-1}) M_x	&			 \text{if $\h(sx) < \h(x)$} \\
vM_{x} &							 \text{if $\h(sx)= \h(x)$.}
\end{cases}
\]

\item[(b)]
There is a unique $\H$-module structure on $\cN$ such that
for all  $s \in S$ and $x \in X$
\[
H_s N_x 
= 
\begin{cases}
N_{sx}&  							 \text{if $\h(sx)>\h(x)$} \\ 
N_{sx}  +  (v-v^{-1}) N_x	&			 \text{if $\h(sx) < \h(x)$} \\
  -v^{-1}N_{x} &							 \text{if $\h(sx)= \h(x)$.}
\end{cases}
\]
\een
\end{theorem}

\begin{remark} Our notation, which is patterned on Soergel's conventions in \cite[\S3]{Soergel}, translates to that of \cite{RV} on setting $H_s= v^{-1}T_\pm(s)$ and
$M_x\text{ (respectively, $N_x$)} = v^{-\h(x)} T(x)$.
\end{remark}

Note that the $\H$-modules $\cM$ and $\cN$ are identical if $sx \neq x$ for all $s \in S$ and $x \in X$. Note that this occurs if and only if $w \in W$ has even length whenever $wx=x$ for some $x \in X$. Following Rains and Vazirani \cite[Definition 3.4]{RV}, we say that a $W$-set $X$ is \emph{even} if it has these equivalent properties.

Denote by $A_1=\langle s_0\rangle $ the unique Coxeter group with a single simple generator $s_0$. Identifying $s_0$ with the nontrivial permutation of $\{1,2\}$ gives an isomorphism $A_1 \cong S_2$. We view the product group $W\times A_1$ 
as a Coxeter group relative to the generating set $S \cup \{ s_0\}$.
 In \cite[\S3]{RV}, Rains and Vazirani describe a construction which attaches to any scaled $W$-set   
$(X,\h)$ an even scaled $W\times A_1$-set $(\wt X, \wt \h)$, with the property that $(X,\h)$ is quasiparabolic if and only if $(\wt X, \wt \h)$ is quasiparabolic. Following \cite{RV}, we refer to $(\wt X, \wt \h)$ as the \emph{even double cover} of $(X,\h)$.
This construction is useful for reducing certain arguments  to the even case, and so we review it  here.

Fix a scaled $W$-set $(X,\h)$ and define $\wt X = X \times \FF_2$, where $\FF_2$ is the  field of two elements viewed as the set  $\{0,1\}$ with addition computed modulo 2.
The groups  $W$ and $A_1$ each act on $\wt X$ by 
\[ w : (x,k) \mapsto (wx, k+\ell(w))\qquand s_0 : (x,k) \mapsto (x,k+1)\]
for $w \in W$ and $x \in X$ and $k \in \FF_2$.
These actions commute with each other and so define an action of $W \times A_1$ on $\wt X$.
Define a height function $\wt \h$ on $\wt X$ by the formula
\[ \wt \h(x,k) = \h (x) + \begin{cases} 0 &\text{if $\h(x) \equiv k \modu 2)$} \\ 1 &\text{if $\h(x) \not \equiv k \modu 2)$.}
\end{cases}
\]
Observe that if $x_0 \in X$ is $W$-minimal if and only if $(x_0, \h(x_0)) \in \wt X$ is $W\times A_1$-minimal.
The following result appears as \cite[Theorem 3.5]{RV}.
\begin{theorem}[Rains and Vazirani \cite{RV}]
If $(X,\h)$ is a scaled $W$-set then $(\wt X,\wt \h)$ is an even scaled $W\times A_1$-set, which is quasiparabolic if and only if $(X,\h)$ is quasiparabolic.
\end{theorem}

Remember that if $(X,\h)$ is quasiparabolic then the $\H(W\times A_1, S\cup \{s_0\})$-modules $\cM(\wt X,\wt\h)$ and $\cN(\wt X,\wt\h)$ are  isomorphic by construction. In the following lemma, which appears as \cite[Proposition 7.7]{RV}, 
note that
$H_{s_0}$ is the generator of $\H(A_1,\{s_0\}) \subset \H(W\times A_1, S\cup \{s_0\})$.

\begin{lemma}[Rains and Vazirani \cite{RV}] \label{submodule-lem}
Suppose every orbit in $X$ contains a $W$-minimal element.
The $\cA$-linear maps   
$\cM(X,\h) \to \cM(\wt X, \wt \h)$ and $ \cN(X,\h) \to \cM(\wt X, \wt \h)$
 with 
\[M_x \mapsto  (H_{s_0}+v^{-1}) M_{(x,\h(x))}
\qquand 
 N_x \mapsto (H_{s_0}-v) M_{(x,\h(x))}
\qquad\text{for $x \in X$}\]
are then injective homomorphisms of $\H(W,S)$-modules.
\end{lemma}

\section{Bar operators, canonical bases, and $W$-graphs}\label{s3-sect}

Everywhere in this section $(W,S)$ is an arbitrary Coxeter system; $\H =\H(W,S)$ is its Iwahori-Hecke algebra; $(X,\h)$ is a fixed quasiparabolic $W$-set; and $\cM=\cM(X,\h)$ and $\cN =\cN(X,\h)$ are the corresponding $\H$-modules defined by Theorem \ref{module-thm}.

\subsection{Bar operators}

We write $f \mapsto \overline {f}$ for the ring involution of $\cA=\ZZ[v,v^{-1}]$ with $v \mapsto v^{-1}$.
A map $U \to V$ of $\cA$-modules is \emph{$\cA$-antilinear} if $x \mapsto y$ implies $ax \mapsto \overline{a} y$ for all $a \in \cA$.
Recall that we also use the notation
$f \mapsto \overline f$ to denote the \emph{bar operator} of $\H$
defined the beginning of the introduction.


\begin{definition}\label{barop-def}
A $\ZZ$-linear map $\cM \to \cM$, denoted $M \mapsto \overline{M}$, is a \emph{bar operator} if 
\[ \overline{HM}= \overline{H} \cdot \overline{M} \qquand \overline{M_{x_0}}
= M_{x_0}\]
for all $(H,M) \in \H \times \cM$ and all $W$-minimal $x_0 \in X$. An $\cA$-antilinear map $\cN \to \cN$ is a \emph{bar operator} if the same conditions hold, \emph{mutatis mutandis}.
\end{definition}


Although at this point there is no obvious obstruction to the modules $\cM$ and $\cN$ each having multiple bar operators, we will nevertheless always denote such maps by the notation $X \mapsto \overline{X}$. We will soon see that in the case which interest us, if a bar operator exists then it is unique, which justifies this convention.

All of our results  concern  quasiparabolic $W$-sets
%
%
%
 whose orbits each contain a (unique) $W$-minimal element. 
Without loss of generality, we can always assume that the height function on such a $W$-set has values all greater than some fixed number, since it makes no difference to translate the height function by a constant on any given orbit. We therefore refer to quasiparabolic $W$-sets whose orbits all have $W$-minimal elements as those which are \emph{bounded below}. 


Assume $(X,\h)$ is bounded below.    
 The set $\cR(x)\subset W$ given by \eqref{cR(x)} is then well-defined for all $x \in X$,
and if  $x_0 \in X$ is the $W$-minimal element in the orbit of $x$,
then
  $H_w M_{x_0} = M_x$ and $H_w N_{x_0} = N_x$ for all $w \in \cR(x)$.
Therefore, if  modules $\cM$ and $\cN$  have bar operators, then 
\be\label{barop-form} \overline {M_x} = \overline{H_w} M_{x_0} 
\qquand
\overline {N_x} = \overline{H_w} N_{x_0}\qquad\text{for any $w \in \cR(x)$.}
\ee
The right sides of these formulas are defined unambiguously once we fix a choice of $w \in \cR(x)$. Since the bar operator of $\H$ is an involution, this implies the following:

\begin{proposition}\label{barop-prop1}
Assume $(X,\h)$ is bounded below. 
\ben
\item[(a)] If $\cM$ (respectively, $\cN$)  has a bar operator, then it is unique.
\item[(b)] If $\cM$ (respectively, $\cN$) has a (unique) bar operator, then it is an involution.
\een
\end{proposition}

While \eqref{barop-form} explicitly describes what the bar operators on $\cM$ and $\cN$ must be if they exist, it is difficult  to show that these formulas are well-defined. 
%
We can show the following, however. 

\begin{theorem}\label{lesselem-prop} Assume $(X,\h)$ is bounded below. Then $\cM$ and $\cN$ both have bar operators if the $\cA$-antilinear maps defined by \eqref{barop-form} are well-defined, in the sense that  
\be\label{sense-eq} \overline{H_a} M_{x_0} = \overline{H_b} M_{x_0}
\qquand \overline{H_a} N_{x_0} = \overline{H_b} N_{x_0}\ee
whenever $x_0 \in X$ is $W$-minimal and  $a,b \in \cR(x)$ for some $x \in Wx_0$.
\end{theorem}

\begin{proof}
First, we will show that the result holds in the case when $(X,\h)$ is an even quasiparabolic set. We will then prove that if the module attached to the even double cover $(\wt X,\wt\h)$ of $(X,\h)$ admits a bar operator, then the modules $\cM$ and $\cN$ each admit bar operators as well. Finally, we will check that \eqref{sense-eq} holds for $(X,\h)$ only if the analogous condition holds for $(\wt X,\wt \h)$.

For the first step, assume that $(X,\h)$ is even (so that $\cM=\cN$) and that there exists a well-defined $\cA$-antilinear map $\cM \to \cM$, to be denoted $M \mapsto \overline{M}$, such that if $x $ belongs to the orbit of the $W$-minimal element $x_0\in X$, then  $\overline{M_x} = \overline{H_w} M_{x_0}$ for any $w \in \cR(x)$.
Clearly $\overline{M_{x_0}} = M_{x_0}$ if $x_0 \in X$ is $W$-minimal since then $\cR(x_0) = \{1\}$, so to show that this map is a bar operator, it remains just to check that $\overline{H_s} \cdot \overline{M_x} = \overline{H_s M_x}$ for $s \in S$ and $x \in X$.
Let $x_0 \in X$ be the $W$-minimal element in the orbit of $x$ and choose $w \in \cR(x)$. If $\h(sx) > \h(x)$ so that $H_s M_x = M_{sx}$, then clearly $\ell(sw) > \ell(w)$ and $sw \in \cR(sx)$, so we have $H_s H_w = H_{sw}$ and \[\overline{H_s} \cdot \overline{M_x} = \overline{H_s} \cdot \overline{H_w} \cdot M_{x_0}=
\overline{H_sH_w} M_{x_0} = \overline{H_{sw}} M_{x_0}= \overline{ M_{sx}} = \overline{H_s M_x}.\]
If $\h(sx) < \h(x)$, then $M_{x} = H_s M_{sx}$ so $\overline{M_x} = \overline{H_s} \cdot \overline{M_{sx}}$ by what we have just shown, and hence
\[ \overline{H_s} \cdot \overline{M_x} = \overline{H_s}\cdot \overline{H_s}\cdot \overline{M_{sx}} = \overline{H_1 + (v-v^{-1})H_s} \cdot \overline{M_{sx}}
=\overline{M_{sx} + (v-v^{-1})M_x}
 = \overline{H_s M_x}.
\]
Since $(X,\h)$ is even, this suffices to show that $\cM=\cN$ has a bar operator.

For the second part of the proof, 
suppose the $\H(W\times A_1, S\cup \{s_0\})$-module $\cM(\wt X,\wt \h)$ admits a bar operator, defined with respect to the quasiparabolic set $(\wt X,\wt \h)$. 
Lemma \ref{submodule-lem} shows that $\cM$ and $\cN$ may be identified with $\H$-submodules of $\cM(\wt X,\wt \h) $,
and we claim that the bar operator on the latter module restricts (via these identifications) to bar operators on $\cM$ and $\cN$.
This   is straightforward to prove after noting that
$H_{s_0} + v^{-1}$ and $H_{s_0}-v$ are bar invariant elements of $\H(W\times A_1,S\cup \{s_0\})$ which commute with all elements of the subalgebra $\H(W,S)$, and also that $x \in X$ is $W$-minimal if and only if $(x,\h(x)) \in \wt X$ is $W\times A_1$-minimal. We omit the details. 

Finally suppose \eqref{sense-eq} holds.
 Fix a $W$-minimal element $x_0 \in X$ and let $x$ belong to its orbit, and write $\tilde M_{x_0}$ and $\tilde N_{x_0}$ for the images of $M_{x_0}\in \cM$ and $N_{x_0}\in\cN$ in $\cM(\wt X ,\wt \h)$ under the homomorphisms in Lemma \ref{submodule-lem}. Observe that $\tilde M_{x_0} - \tilde N_{x_0}= (v+v^{-1}) M_{(x_0,\h(x_0))}$ and note that if $k \in \FF_2$ then 
$\mathcal{R}_{\wt\h}(x,k) =\left\{  s_0^e\cdot w : w \in \cR(x)\right\}$
where $e \in\{0,1\}$ is such that $e\equiv k-\h(x) \modu 2)$.
Hence, if $a,b \in \mathcal{R}_{\wt\h}(x,k)$,
then $a=s_0^e\cdot a'$ and $b=s_0^e\cdot b'$ for some $a',b' \in \cR(x)$,
so by Lemma \ref{submodule-lem} 
\[ 
\ba
(v+v^{-1}) \overline{H_{a}} M_{(x_0,\h(x_0))} &= 
\overline{H_{s_0}^e} \cdot  \overline{H_{a'}}  \tilde M_{x_0}
-
\overline{H_{s_0}^e} \cdot  \overline{H_{a'}}  \tilde N_{x_0}
\\
&
=   \overline{H_{s_0}^e} \cdot  \overline{H_{b'}}  \tilde M_{x_0}
-
\overline{H_{s_0}^e} \cdot  \overline{H_{b'}}  \tilde N_{x_0}
=(v+v^{-1}) \overline{H_{b}} M_{(x_0,\h(x_0))}.
\ea
\]
From this, we conclude that
$\overline{H_{a}} M_{(x_0,\h(x_0))} = \overline{H_{b}} M_{(x_0,\h(x_0))} $, which is what we wanted to show. 
\end{proof}

As an application,
we  recover the following result of Deodhar from \cite[\S2]{Deodhar}.

\begin{corollary}[Deodhar \cite{Deodhar}] \label{freebar-cor} If  $(X,\h) = (W^J,\ell)$ for some $J\subset S$, 
then the corresponding $\H$-modules $\cM=\cM(X,\h)$ and $\cN=\cN(X,\h)$ both admit unique bar operators.
\end{corollary}

\begin{proof}
The condition in Theorem \ref{lesselem-prop} holds trivially since  $\cR(x) = \{x\}$ for all $x \in W^J$.
\end{proof}

\begin{remark}\label{freebar-rmk}
We do not know of any examples of transitive, bounded quasiparabolic sets $(X,\h)$ for which $\cR(x)$ is always a singleton set as in the preceding proof except those isomorphic to $(W^J,\ell)$ for some $J\subset S$;  see Conjecture \ref{parabolic-conj}.
\end{remark}

%

Recall that $\leq$ denotes the Bruhat order on $(X,\h)$, as given in Definition \ref{bruhat-def}.

\begin{lemma}\label{barop-lem1}
Assume $(X,\h)$  is bounded below and let $x \in X$.
\ben

\item[(a)] If $\cM$ has a bar operator $M \mapsto \overline{M}$  then $\overline{M_x}  \in M_x + \cA \spanning\{ M_w : w <x\}$.
\item[(b)] If  $\cN$ has a bar operator $N \mapsto \overline{N}$ then $\overline{N_x}  \in N_x + \cA \spanning\{ N_w : w <x\}$.
\een
In particular, when defined,  $\left\{\overline{M_x}\right\}_{x \in X}$ and $\left\{ \overline{N_x} \right\}_{x \in X}$ are $\cA$-bases of $\cM$ and $\cN$, respectively.
\end{lemma}

\begin{proof}
We only prove part (a), as the proof of (b) is identical. 
  If $x$ is $W$-minimal then the desired containment is automatic, so assume $x$ is not $W$-minimal and that $\overline{M_{x'}} \in M_{x'} + \cA\spanning\{M_w: w<x'\}$ for all  $x'<x$ in $X$. There is then $s \in S$ such that $\h(sx)<\h(x)$ (by the definition of $W$-minimal), so, using that $M_x = H_s M_{sx}$ and the inductive hypothesis, we have
  \[ \overline{M_x }= \overline{H_s}\cdot \overline{M_{sx}}
   \in \overline{H_s}  \Bigl(M_{sx} + \sum_{w<sx} \cA \cdot   M_w\Bigr)
\subset
M_x + \cA \spanning\{ M_{sx} \}+  \cA \spanning\left\{ \overline{H_s} M_w : w < sx\right\}.
  \]
Since $\overline{H_s} = H_s + v^{-1}-v$ we have 
$ \overline{H_s} M_{w} \in \cA \cdot M_w + \cA \cdot M_{sw}$ for all $w \in X$. Thus $\overline{M_x}$ has the desired unitriangular form provided that whenever $w \in X$ such that $w<sx<x$ we have $sw<x$; this  property holds by Lemma \ref{bruhat-lem}.
Finally, since all lower intervals in the poset $(X,\leq)$ are finite, it follows from (a) and (b) that $\left\{\overline{M_x}\right\}_{x \in X}$ and $\left\{ \overline{N_x} \right\}_{x \in X}$ are $\cA$-bases of $\cM$ and $\cN$. 
\end{proof}

Write 
$\Theta: \H \to \H$ for the $\cA$-linear with
$\Theta(H_w) = (-1)^{\ell(w)} \overline{H_w}$ for $w \in W$; one checks that $\Theta$ is an $\cA$-algebra automorphism.
Next, define \[\hmin(x) =\min_{w \in W} \h(wx)\qquad\text{for $x \in X$}.\] Note that $\hmin(x) =\h(x_0)$ if there exists a  $W$-minimal element $x_0$  in the orbit of $x$, and that otherwise $\hmin(x)$ is undefined.
Finally, when $(X,\h)$ is bounded below and $\cN$ (respectively, $\cM$) has a bar operator, 
we define 
$\Phi_{\cM\cN}:\cM \to \cN
$ and $ 
\Phi_{\cM\cN}:\cM \to \cN$  as  the  $\cA$-linear maps with
\be\label{phi} \Phi_{\cM\cN}(M_x) = (-1)^{\h(x)-\hmin(x)} \overline{ N_x} \qquand \Phi_{\cN\cM}(N_x) = (-1)^{\h(x)-\hmin(x)} \overline{M_x}\ee
for $x \in X$. 
%
These maps are ``$\Theta$-twisted homomorphisms'' of $\H$-modules in the following  sense.

\begin{lemma}\label{barop-lem2}
Assume $(X,\h)$  is bounded below. 
When respectively defined, the maps $\Phi_{\cM\cN}$ and $\Phi_{\cN\cM}$ are  bijections such that for all $(H, M , N) \in \H\times \cM\times \cN$ it holds that 
\[ \Phi_{\cM\cN}(HM) = \Theta(H)   \Phi_{\cM\cN}(M)\qquand \Phi_{\cN\cM}(HN) = \Theta(H)   \Phi_{\cN\cM}(N).\] 
\end{lemma}

\begin{proof}
Both $\Phi_{\cM\cN}$ and $\Phi_{\cN\cM}$ are bijections since, by the previous lemma, they each map a basis to a basis. 
Since 
 $\Theta$ is an algebra automorphism,  to show that $\Phi_{\cM\cN}$ has the desired property   it is enough to check that $\Phi_{\cM\cN}(H_s  M_x) = -(-1)^{\h(x)-\hmin(x)} \overline{H_s N_x}=\Theta(H_s) \Phi_{\cM\cN}(M_x)$ for $s \in S$ and $x \in X$. This is straightforward;
for example, if $\h(sx) = \h(x)$ then 
\[
\Phi_{\cM\cN}(H_s  M_x) =\Phi_{\cM\cN}(v M_x) =   
 (-1)^{\h(x)-\hmin(x)}  \overline{ v^{-1}N_x}  
= -(-1)^{\h(x)-\hmin(x)}   \overline{ H_sN_x} .\]
The calculations in the case when $\h(sx) > \h(x)$ and  $\h(sx) < \h(x)$ are similar.
An identical argument shows that the same property holds for $\Phi_{\cN\cM}$.
\end{proof}

We now prove the following conceptually plausible, but technically nontrivial result.

\begin{proposition}\label{barop-prop2} Assume $(X,\h)$ is bounded below. If either of the corresponding $\H$-modules $\cM$ or $\cN$ has a bar operator, then both modules have unique bar operators.
\end{proposition}

\begin{proof}
Assume that $\cM$ has a bar operator; we will show that this implies that $\cN$ does as well. The converse implication holds by a symmetric argument.
Let $\Phi_{\cM\cN} : \cM \to \cN$ be the map given before Lemma \ref{barop-lem2}, 
and
define $N \mapsto \overline{N}$ as the $\cA$-antilinear map $\cN \to \cN$ with 
\[ \overline{N_x} =  (-1)^{\h(x)-\hmin(x)} \Phi_{\cN\cM}^{-1}(M_x)\qquad\text{for $x \in X$}.\]
We check that this map has the defining properties of a bar operator. If $x \in X$ is $W$-minimal then $M_x = \overline{M_x}$ so by  definition  $\overline{N_x} = N_x$. In turn, if $s \in S$ and $x \in X$ then we claim that
\[\ba
 \Phi_{\cN\cM}\(\overline{H_s}\cdot  \overline{N_x}\) &= -H_s \cdot\Phi_{\cN\cM}(\overline{N_x}) 
= \overline{ \Theta(H_s) \Phi_{\cN\cM}(N_x)}
= \overline{  \Phi_{\cN\cM}(H_sN_x)}
=\Phi_{\cN\cM}(\overline{H_s N_x})
.
\ea
\]
To check this, observe that the 
 first and third equalities hold by Lemma \ref{barop-lem2}; the second holds by definition since the bar operator on $\cM$ is an involution; and the last equality holds since by construction $\Phi_{\cN\cM}(\overline{N}) = \overline{\Phi_{\cN\cM}(N)}$ for all $N \in \cN$.
As $\Phi_{\cN\cM}$ is a bijection,
we conclude that $\overline{H_s N_x} = \overline{H_s}\cdot \overline{ N_x} $.
Since the bar operator on $\H$ is a ring involution,
this suffices 
 to show that
 $\overline{H_w} \cdot \overline{N_x} = \overline{H_w N_x}$ for all $w \in W$ and $x \in X$.
We deduce by antilinearity that $\overline{HN} = \overline{H}\cdot \overline{N}$ for all $H \in \H$ and $N \in \cN$. Hence the map $N \mapsto \overline{N}$ is a bar operator on $\cN$, as desired.
\end{proof}

Given a quasiparabolic $W$-set $(X,\h)$ which is bounded below, we say that $(X,\h)$ \emph{admits a bar operator} if both (equivalently, either) of the modules $\cM$ and $\cN$ have a (unique) bar operator.

 \begin{remark}\label{precanon-remark}
 Assume $(X,\h)$ is bounded below and  admits a bar operator.
 Let 
 \[V = \cM \text{ (respectively, $\cN$)}\qquand
 a_c = M_c \text{ (respectively, $N_c$) for $c \in X$.}\]
 Also
 define $\langle -,-\rangle : V \times V \to \cA$ as the $\cA$-sesquilinear inner product with 
 \[\langle a_c, \overline{a_c'}\rangle = \delta_{c,c'}\qquad\text{for $c,c' \in X$.}\] Combining Definition \ref{barop-def}, Proposition \ref{barop-prop1}, and Lemma \ref{barop-lem1} shows that 
the bar operator on $V$ together with the inner product $\langle-,-\rangle$ and the ``standard basis'' $\{a_c\}_{c \in X}$ are  
what Webster \cite{Webster} calls a \emph{pre-canonical structure}. 
When $(X,\h) = (W,\ell)$, this
pre-canonical structure  arises from a ``categorification'' of the Iwahori-Hecke algebra, via the theory of either intersection cohomology or Soergel bimodules;  see \cite{EW,EW2,Soergel90,Soergel07}. It would be interesting to have an interpretation along these lines for the pre-canonical structure attached to a general quasiparabolic $W$-set.
\end{remark}
 
The following statement is clear from the proof of Theorem \ref{lesselem-prop}.
 
\begin{corollary}
Assume $(X,\h)$ is bounded below. The quasiparabolic set $(X,\h)$ then admits a bar operator if and only if its even double cover $(\wt X,\wt \h)$ also admits a bar operator.
\end{corollary}



When $(X,\h)$ is 
the $W$-set of left cosets of a standard parabolic subgroup (see 
Example \ref{parabolic-ex}), 
the following proposition reduces to
the main result of Deodhar's paper \cite[Theorem 2.1]{Deodhar2}.

\begin{proposition}\label{barop-prop3}
Assume the quasiparabolic $W$-set $(X,\h)$ is bounded below and admits a bar operator, so that the maps $\Phi_{\cM\cN}$ and $\Phi_{\cN\cM}$ are both  defined.
\ben
\item[(a)]  The following  diagrams  commute:
\[
\begin{diagram}
\cM & \rTo^{M\mapsto \overline{M}} &\cM \\
\dTo^{\Phi_{\cM\cN}}&  & \dTo_{\Phi_{\cM\cN}} \\
\cN  &\rTo^{N\mapsto\overline{N}}  & \cN
\end{diagram}
\qquad
\qquad
\begin{diagram}
\cM & \rTo^{M\mapsto \overline{M}} &\cM \\
\uTo^{\Phi_{\cN\cM}}&  & \uTo_{\Phi_{\cN\cM}} \\
\cN  &\rTo^{N\mapsto\overline{N}}  & \cN
\end{diagram}
\]

\item[(b)] 
The following diagrams  commute:
\[
\begin{diagram}
\cM & \rTo^{M\mapsto \overline{M}} &\cM \\
\dTo^{\Phi_{\cM\cN}}&  & \uTo_{\Phi_{\cM\cN}^{-1}} \\
\cN  &\rTo^{N\mapsto\overline{N}}  & \cN
\end{diagram}
\qquad
\qquad
\begin{diagram}
\cM & \rTo^{M\mapsto \overline{M}} &\cM \\
\uTo^{\Phi_{\cN\cM}}&  & \dTo_{\Phi_{\cN\cM}^{-1}} \\
\cN  &\rTo^{N\mapsto\overline{N}}  & \cN
\end{diagram}
\]

\item[(c)]  The maps $\Phi_{\cM\cN}$ and $\Phi_{\cN\cM}$ are  inverses of each other.

\een
\end{proposition}

\begin{proof}
Parts (a) and (b) follow  from Proposition \ref{barop-prop1} and the definitions of $\Phi_{\cN\cM}$ and $\Phi_{\cM\cN}$, while part (c) follows   from the definitions and part (a). 
%
\end{proof}

\subsection{Canonical bases}\label{cb-sect}

Everywhere in this section we assume that $(X,\h)$ is a quasiparabolic $W$-set which is bounded below and admits a bar operator; $\cM= \cM(X,\h)$ and $\cN=\cN(X,\h)$ are as in Theorem \ref{module-thm}. 
%
We begin by recalling  the following well-known theorem of Kazhdan and Lusztig \cite{KL}:
\begin{theorem}[Kazhdan and Lusztig \cite{KL}] \label{kl-thm}
For  each $w \in W$ there is a unique   $\underline H_w \in \H$ with
\[ \underline H_w = \overline{\underline H_w} \in H_w + \sum_{ y<w} v^{-1} \ZZ[v^{-1}]\cdot H_y.\]
The elements $\{ \underline H_w\}_{w \in W}$ form an $\cA$-basis for  $\H$,   called the \emph{Kazhdan-Lusztig basis}.
\end{theorem}

One checks that 
$ \underline H_1 = H_1 = 1 $ and $ \underline H_s   = H_s + v^{-1}$ for $s \in S$.
Define $h_{x,y} \in \ZZ[v^{-1}]$ for $x,y \in W$  such that 
$\underline H_y  =  \sum_{x \in W} h_{x,y} H_y.$
It follows by recent work of Elias and Williamson \cite{EW}  that the polynomials $h_{x,y}$ actually always belong to $\NN[v^{-1}]$.
%
Moreover, 
when $W$ is the Weyl group of a complex semisimple Lie algebra, 
these 
polynomials encode in a certain precise sense the multiplicities of simple modules in Verma modules in the principal block of category $\cO$; this is the original \emph{Kazhdan-Lusztig conjecture} \cite[Conjecture 1.5]{KL}.

Such phenomena suggest that it would be interesting to formulate an analogue of the Kazhdan-Lusztig basis for the modules $\cM$ and $\cN$. 
For this purpose, we will need the
following   lemma:

\begin{lemma}\label{barinv-lem}
Let $\cC\subset \cA$ be a subset  such that $\{ f \in \cC : \overline f =f\} = \{0\}$;  for example, $v^{-1}\ZZ[v^{-1}]$. Then 0 is the only element of $\cM$ (respectively, $\cN$) which is both 
(i) invariant under the  bar operator 
and (ii) a linear combination of standard basis elements with coefficients all in $\cC$.
\end{lemma}

\begin{proof}
Let $\varepsilon_x \in \cC$ for $x \in X$ be such that the element $\varepsilon = \sum_{x \in X} \varepsilon_x M_x$ (respectively, $\sum_{x \in X} \varepsilon_x N_x$) has properties (i) and (ii). Suppose $\varepsilon \neq 0$; we argue by contradiction. Let $x$ be maximal in $(X,\leq)$
such that $\varepsilon_x \neq 0$.  
By Lemma \ref{barop-lem1},
the coefficient of $M_x $ (respectively, $N_x$) in $\overline{\varepsilon}$ is then $\overline{\varepsilon_x}$, so since $\overline{\varepsilon} = \varepsilon$ we must have $\overline{\varepsilon_x} = \varepsilon_x$; our hypothesis on $\cC$ now leads to the contradiction $\varepsilon_x = 0$. 
\end{proof}

The following generalizes  Theorem \ref{kl-thm} 
and also results of Deodhar from \cite[\S3]{Deodhar}. 

\begin{theorem}\label{qpcanon-thm} 
Assume the quasiparabolic $W$-set $(X,\h)$  is bounded below and admits a bar operator. 
For each $x \in X$ there are  unique elements $\underline M_x \in \cM(X,\h)$ and $\underline N_x \in \cN(X,\h)$ with 
\[ \underline M_x = \overline{\underline M_x} \in M_x + \sum_{w<x} v^{-1}\ZZ[v^{-1}] \cdot M_w
\qquand
 \underline N_x = \overline{\underline N_x} \in N_x + \sum_{w<x} v^{-1}\ZZ[v^{-1}] \cdot N_w\]
 where both sums are over $w \in X$.
The elements $\{ \underline M_x\}_{x \in X}$ and $\{ \underline N_x\}_{x\in X}$ form  $\cA$-bases for $\cM(X,\h)$  $\cN(X,\h)$, which we refer to as the \emph{canonical bases} of these modules.
\end{theorem}

\begin{proof}
The theorem follows from the general fact (first proved using different terminology by Du \cite{Du}) that any pre-canonical structure whose index set $(X,\leq)$ has finite lower intervals admits a unique canonical basis; compare Remark \ref{precanon-remark} with \cite[Theorem 2.5]{EM3}.
For a  self-contained proof,
one can adapt, almost verbatim, the argument which Soergel gives to prove \cite[Theorem 3.1]{Soergel}.
%
%
\end{proof}

Define $m_{x,y}$ and $n_{x,y}$ for $x,y \in X$ as the polynomials in $\ZZ[v^{-1}]$ such that
\be\label{qpklpol}\underline M_y = \sum_{x \in X} m_{x,y} M_x
\qquand
\underline N_y = \sum_{x \in X} n_{x,y} N_x .
\ee
Let
$ \mu_m(x,y) $ and $ \mu_n(x,y)$
denote the coefficients of $v^{-1}$ in $m_{x,y}$ and $n_{x,y}$ respectively. 
Observe that if $x<y$ then $m_{x,y}$ and $n_{x,y}$ are both polynomials in $v^{-1}$ without constant term, while if $x \not < y$ then $m_{x,y} = n_{x,y} = \delta_{x,y}$. When $(X,\h) = (W,\ell)$ as in Example \ref{case0-ex}, we have $m_{x,y} = n_{x,y} = h_{x,y}$.


\begin{remark}
A surprising property of the polynomials $h_{x,y}$ is that their coefficients are always nonnegative \cite{EW}.
By contrast,  $m_{x,y}$ and $n_{x,y}$ can each have both positive and negative coefficients.
If $(X,\h) = (W^J,\ell)$ for some $J\subset S$ as in Example \ref{parabolic-ex},
then   $\{m_{x,y}\}\subset \{h_{x,y}\}\subset \NN[v^{-1}]$ (see \cite [Proposition 3.4]{Deodhar}), but even in this   case the polynomials $n_{x,y}$ may still have negative coefficients.
\end{remark}

The following theorem describes the action of $\H$ on the  basis elements $\underline M_x$  and $\underline N_x$.

\begin{theorem}\label{M-thm} Let $s \in S$ and $x \in X$. 
Recall that $\underline H_s = H_s + v^{-1}$.
\ben
\item[(a)] In $\cM$, the following multiplication formula holds:
\[ \underline H_s \underline M_x = \begin{cases} (v+v^{-1}) \underline M_x & \text{if }\h(sx) \leq \h(x)
\\[-10pt]\\
\underline M_{sx} +  \sum_{sw\leq w< x} \mu_m(w,x) \underline M_w&\text{if }\h(sx)> \h(x).
\end{cases}
\]

\item[(b)] In $\cN$, the following multiplication formula holds:
\[
\underline H_s \underline N_x = \begin{cases} (v+v^{-1}) \underline N_x & \text{if }\h(sx) < \h(x)
\\[-10pt]\\
\underline N_{sx} +  \sum_{sw<w< x} \mu_n(w,x) \underline N_w&\text{if }\h(sx)> \h(x)
\\[-10pt]\\
\sum_{sw<w<x} \mu_n(w,x) \underline N_w&\text{if }\h(sx)=\h(x).
\end{cases}
\]

\een
\end{theorem}

\begin{proof}
We first prove part (a); there are three cases to consider.
First suppose $\h(sx)>\h(x)$. Using the definition the module $\cM$ in Theorem \ref{module-thm}, one checks that the linear combination
$ \underline H_s \underline M_x- \underline M_{sx}  - \sum_{sw\leq w <x} \mu_m(w,x) \underline M_w$ is bar invariant and belongs to $v^{-1}\ZZ[v^{-1}]\spanning\{ M_w: w \in X\}$, so it must be zero by Lemma \ref{barinv-lem}.
  
 Next suppose $sx=x$. 
The following identity then holds, since one can check that the difference between the two sides is a bar invariant linear combination of standard basis elements $M_w$ with coefficients in $v^{-1}\ZZ[v^{-1}]$, and the only such element is zero by Lemma \ref{barinv-lem}:
\be\label{vvinv} \underline H_s \underline M_x = (v+v^{-1}) \underline M_x -\sum_{sw>w<x} \mu_m(w,x) \underline M_w.\ee Since $\underline H_s\underline H_s = (v+v^{-1})\underline H_s$,  multiplying both sides of this equation by $\underline H_s$ 
implies that
 \[ \sum_{sw>w<x} \mu_m(w,x) \underline H_s \underline M_w 
 = 0.\] 
By considering those $w \in X$ which are maximal in the Bruhat order 
such $sw>w<x$ and $\mu_m(w,x)\neq 0$, and then expanding the products $\underline H_s \underline M_w$,
it becomes clear that
 the preceding equation can only hold if $\mu_m(w,x) = 0$ for all $w \in X$ with $sw>w<x$. We   conclude from \eqref{vvinv} that $\underline H_s \underline M_x= (v+v^{-1})\underline M_x$ when $sx=x$.

 Finally suppose $\h(sx) < \h(x)$. What we have already shown implies
$ \underline M_x = \underline H_s \underline M_{sx} - \sum_{sw\leq w < sx} \mu_m(w,sx) \underline M_w.$
Since $\underline H_s \underline H_s = (v+v^{-1}) \underline H_s$, 
we obtain by induction 
\[\underline H_s \underline M_x = (v+v^{-1})\underline H_s \underline M_{sx} - \sum_{sw\leq w <s x} \mu_m(w,sx) \underline H_s\underline M_w= (v+v^{-1}) \underline M_x.\]
  This completes the proof of part (a).
  
  The proof of part (b) is similar.
  The formula for $\underline H_s \underline N_x$ when $\h(sx) \neq \h(x)$ follows by arguments similar to the ones already given.
  When $sx=x$, one checks   that $\underline H_s \underline N_x - \sum_{sw<w<x} \mu_n(w,x) \underline N_w$ is a bar invariant element of $v^{-1}\ZZ[v^{-1}]\spanning\{ N_w : w \in X\}$, and hence zero by Lemma \ref{barinv-lem}.
\end{proof}

Define $\wt m_{x,y} = v^{\h(y)-\h(x)}$ and $\wt n_{x,y} = v^{\h(y)-\h(x)} n_{x,y}$ for $x,y \in X$.
The preceding theorem translates to the following  recurrences, which one can use to  compute these polynomials.

\begin{corollary}\label{M-cor}
Let $x,y \in X$ and $s \in S$. 
\ben
\item[(a)]
If $sy=y$ then $\wt m_{x,y} = \wt m_{sx,y}$
and if $sy < y$ then
\[ 
\wt m_{x,y} = \wt m_{sx,y}= \left.\begin{cases} \wt m_{x,sy} + v^2\cdot \wt m_{sx,sy} & \text{if }sx>x \\ 
v^2\cdot \wt m_{x,sy}  +  \wt m_{sx,sy} & \text{if }sx\leq x \end{cases}\right\}  
 - \sum_{\substack{ x<t<sy \\ st \leq t }} \mu_m (t,sy)\cdot v^{\h(y)-\h(t)}\cdot \wt m_{x,t}
.\]

\item[(b)] If $sy <y$ then 
\[ 
\wt n_{x,y} = \wt n_{sx,y}=  \left.\begin{cases} \wt n_{x,sy} + v^2\cdot \wt n_{sx,sy} & \text{if }sx>x \\ 
v^2\cdot \wt n_{x,sy}  + \wt n_{sx,sy} & \text{if }sx<x \\ 0 & \text{if }sx=x \end{cases}\right\}  
 - \sum_{\substack{ x<t<sy \\ st < t }} \mu_n (t,sy)\cdot v^{\h(y)-\h(t)}\cdot \wt n_{x,t}.
\]
\een
\end{corollary}


\begin{proof}
The assertion that $\wt m_{x,y} = \wt m_{sx,y}$ if $sy \leq y$ follows by 
comparing the coefficients of $M_x$ in the identity
$\underline H_s \underline M_y = (v+v^{-1})\underline M_y$. The second equality in part (b) follows by comparing coefficients in the identity $\underline H_s \underline M_{sy} = \underline M_{y} + \sum_{st\leq t < y} \mu_m(t,y) \underline M_t$.
The proof of part (c) is similar.
\end{proof}

By definition $m_{x,y} = n_{x,y} = 0$ when $x \not \leq y$.
When $x \leq y$, the following parity property holds:

\begin{proposition}\label{parity-prop}
If $x,y \in X$ with $x\leq y$
then
\[ v^{\h(y)-\h(x)} m_{x,y}  =\wt m_{x,y}  \in 1 + v^2 \ZZ[v^2]
\qquand
v^{\h(y)-\h(x)} n_{x,y} = \wt n_{x,y} \in \ZZ[v^2].
\]
Consequently, $\mu_m(x,y) = \mu_n(x,y) =0$ whenever $\h(y)-\h(x)$ is even.
\end{proposition}

\begin{proof}
If $y$ is $W$-minimal then $x\leq y$ implies $x=y$ in which case $\wt m_{x,y} = \wt n_{x,y} = 1 \in 1 + v^2\ZZ[v^2]$.
Alternatively, suppose $y$ is not $W$-minimal, so that there exists some $s \in S$ such that $sy<y$.
We may assume by induction that $\wt m_{x',y'} $ and $\wt n_{x',y'}$ respectively belong to $1+v^2\ZZ[v^2]$ and $v^2\ZZ[v^2]$ 
for all $x',y'\in X$ with $x' \leq y' < y$.
The coefficients $\mu_m(t,sy)$ and $\mu_n(t,sy)$ are then nonzero only for those $t \in X$ with $\h(y)-\h(t)$ even,  so the recurrences in Corollary \ref{M-cor} imply via Lemma \ref{bruhat-lem} that $\wt m_{x,y} \in 1 + v^2\ZZ[v^2]$ and $\wt n_{x,y} \in v^2 \ZZ[v^2]$.
\end{proof}


%

Finally, we clarify that nothing is gained or lost by preferring the indeterminate $v^{-1}$ over $v$ in Theorem \ref{qpcanon-thm}.
Define for $y\in X$ the elements
\be\label{qpcanon2-eq} \underline M'_y = \sum_{x \in X} (-1)^{\h(y)-\h(x)} \cdot \overline{n_{x,y}} \cdot M_x
\qquand
\underline N'_y = \sum_{x \in X} (-1)^{\h(y)-\h(x)}\cdot  \overline{m_{x,y}} \cdot N_x
.\ee
Write
 $\varepsilon(x)=(-1)^{\h(x)-\hmin(x)}$ for $x \in X$ and recall the definition of the maps $\Phi_{\cM\cN}$ and $\Phi_{\cN\cM}$ from \eqref{phi}.
We note the following lemma.

\begin{lemma}\label{prime-lem}
For each  $x \in X$ it holds that
\[\underline M_x' = \varepsilon(x)\cdot \Phi_{\cN\cM}(\underline N_x)
\qquand
\underline N_x' = \varepsilon(x) \cdot\Phi_{\cM\cN}(\underline M_x)
.\]
\end{lemma}

\begin{proof}
We have $\underline M'_x = \varepsilon(x) \cdot \overline{\Phi_{\cN\cM}(\underline N_x)}$
and $
\underline N_x' = \varepsilon(x) \cdot \overline{\Phi_{\cM\cN}(\underline M_x)}
$
by the definition of the maps $\Phi_{\cM\cN}$ and $\Phi_{\cN\cM}$. Proposition \ref{barop-prop3}(a) shows that these equations remain valid  after erasing the bar operators on the right, since the canonical bases of $\cM$ and $\cN$ are bar invariant.
\end{proof}


Since the maps $\Phi_{\cN\cM}$ and $\Phi_{\cM\cN}$ are bijections,    $\{ \underline M_x'\}_{x \in X}$ and $\{\underline N_x'\}_{x \in X}$ are $\cA$-bases of $\cM$ and $\cN$, respectively.
These bases are uniquely characterized analogously to Theorem \ref{qpcanon-thm}, as follows.

\begin{corollary}\label{M'-cor}
For each $x \in X$,
the elements
$\underline M_x'  $ and $\underline N_x' $ are the unique ones in $\cM$ and $\cN$ with 
\[\underline M_x' = \overline{\underline M_x'} \in M_x + \sum_{w<x} v\ZZ[v] \cdot M_w
\qquand
 \underline N_x' = \overline{\underline N_x'} \in N_x + \sum_{w<x} v\ZZ[v] \cdot N_w.\]
\end{corollary}


\begin{proof}
In view of Lemma \ref{prime-lem}, both $\underline M_x'$ and $\underline N_x'$ are bar invariant by Proposition \ref{barop-prop3}(a), and they are given by  unitriangular linear combinations of standard basis elements of the desired form by definition. The uniqueness of the elements with these properties follows from Lemma \ref{barinv-lem}. 
\end{proof}

\begin{remark}\label{C'-remark}
To conclude this section, we explain  more precisely how our results and notation connect 
to earlier work. Define $T_w= v^{\ell(w)} H_w \in \H$ for $w \in W$. Often, for example in \cite{Deodhar,KL,RV}, formulas involving $\H$ are written in the terms of the basis $\{T_w\}$ rather than $\{H_w\}$.
\begin{itemize}
\item
If $(X,\h) = (W,\ell)$ as in Example \ref{case0-ex}, then $\cM \cong\cN \cong \H$ as left $\H$-modules and $m_{x,y} = n_{x,y} = h_{x,y}$ for all $x,y \in W$. In this case the bases $\{ \underline M_w\} $ and $\{ \underline M'_w\}$  of $\cM$ may be  respectively identified with the bases of $\H$ which are  denoted $\{ C'_w\}$ and $\{ C_w\}$ in \cite{KL}.

\item If $(X,\h) = (W^J,\ell)$ for some $J \subset S$ as in Example \ref{parabolic-ex}, then $\cM $ (respectively, $\cN$) is isomorphic to the $\H$-module $\cM^J$ defined in \cite{Deodhar} with $u=q$ (respectively $u=-1$).
In this case  the  basis which Deodhar denotes $\{ C^J_w\}$ corresponds to the basis $\{ \underline M'_w\}$ (respectively, $\{\underline N'_w\}$).

\end{itemize}
\end{remark}

\subsection{$W$-graphs}
\label{Wgraph-sect}

Recall that $\cA = \ZZ[v,v^{-1}]$. 
Let $\cX$ be an $\H$-module which is free as an $\cA$-module. Given an $\cA$-basis $V\subset \cX$,
 consider the directed graph with vertex set $V$ and with an edge from $x\in V$ to $y \in V$ whenever there exists $H \in \H$ such that the coefficient of $y$ in $Hx$ is nonzero.
Each strongly connected component in this graph spans a quotient $\H$-module since its complement spans a submodule of $\cX$. There is a natural partial order on the set of strongly connected components in any directed graph, and this order in our present context  
gives rise to a filtration of $\cX$.
  For some choices of bases of $V$, this filtration can be interesting and nontrivial.

When this procedure is applied to the Kazhdan-Lusztig basis of $\H$ (viewed as a left module over itself), the graph one obtains has a particular form, which serves as the prototypical example of a \emph{$W$-graph}. 
The notion of a $W$-graph dates to Kazhdan and Lusztig's paper \cite{KL}, but our conventions in the following definitions have been adopted from Stembridge's more recent work \cite{Stembridge,Stembridge2}.

\begin{definition}
Let $I$ be a finite set. 
An \emph{$I$-labeled graph} is a triple $\Gamma = (V,\omega,\tau)$ where
\ben
\item[(i)] $V$ is a finite vertex set;
\item[(ii)] $\omega : V\times V \to \cA$ is a map;
\item[(iii)] $\tau : V \to \cP(I)$ is a map assigning a subset of $I$ to each vertex.
\een
\end{definition}

We write $\omega(x\to y)$ for $\omega(x,y)$ when $x,y \in V$.
One views $\Gamma$ as a weighted directed graph on the vertex set $V$ with an edge from $x$ to $y$ when the weight $\omega(x \to y)$ is nonzero. 

\begin{definition} \label{Wgraph-def}
Fix a Coxeter system $(W,S)$. 
An $S$-labeled graph $\Gamma = (V,\omega,\tau)$ is a \emph{$W$-graph}
if the free $\cA$-module 
generated by $V$ may be given an $\H$-module structure with 
\[  H_s    x = \begin{cases} v    x &\text{if }s \notin \tau(x) \\ -v^{-1}   x +  \ds\sum_{{y \in V ;\hs s \notin \tau(y)}} \omega(x\to y)   y&\text{if }s \in \tau(x)\end{cases}
\qquad\text{for $s \in S$ and $x \in V$.}
\]
\end{definition}

The prototypical $W$-graph defined by the Kazhdan-Lusztig basis of $\H$ has several notable features; Stembridge \cite{Stembridge, Stembridge2} calls $W$-graphs with these features \emph{admissible}.
We introduce the following slight variant of Stembridge's definition.

\begin{definition}\label{admissible-def} An $I$-labeled graph $\Gamma = (V,\omega,\tau)$ is \emph{quasi-admissible} if 
 \ben
\item[(a)] it is \emph{reduced}  in the sense that $\omega(x\to y) = 0$ whenever $\tau(x)\subset \tau(y)$.
\item[(b)] its edge weights $\omega(x\to y)$ are all integers;
\item[(c)] it is bipartite; 
\item[(d)] the edge weights satisfy $\omega(x\to y) = \omega(y\to x)$ whenever  $\tau(x)\not \subset \tau(y)$ and $\tau(y)\not \subset \tau(x)$.
\een
The $I$-labeled graph $\Gamma$ is \emph{admissible} if  its integer edge weights are all nonnegative.
\end{definition}


Let $(X,\h)$ denote a fixed quasiparabolic $W$-set
 which is bounded below and admits a bar operator, so that canonical bases $\{\underline M_x\}\subset \cM=\cM(X,\h)$ and $\{\underline N_x\}\subset \cN=\cN(X,\h)$ given in Theorem \ref{qpcanon-thm} are  well-defined. We show below that these  bases induce two quasi-admissible $W$-graph structures on the set $X$. 
 Define  the maps  $\mu_m, \mu_n: X \times X \to \ZZ$ 
 as just before \eqref{qpklpol}. 
 
\begin{lemma} \label{omega-lem}
Let $x,y \in X$ with $x<y$.
\ben
\item[(a)] If there exists $s \in S$ with $sy \leq y$ and $sx>x$, then $\mu_m(x,y) = \delta_{sx,y}$.

\item[(b)] If there exists $s \in S$ with $sy< y$ and $sx \geq x$, then $\mu_n(x,y) = \delta_{sx,y}$.
\een
\end{lemma}

\begin{proof}
Suppose $s \in S$ is such that $sy \leq y$ (respectively, $sy <y$), so that Corollary \ref{M-cor}  implies  \be\label{thefrm} m_{x,y} = v^{\h(x)-\h(sx)} m_{sx,y}\qquad\text{(respectively, $n_{x,y} = v^{\h(x)-\h(sx)} n_{sx,y}$).}\ee
If $sx=y>x$ then  $m_{x,y} = n_{x,y} = v^{-1}$ so $\mu_m(x,y) = \mu_n(x,y) = 1$.
Suppose alternatively that $sx \neq y$.  Lemma \ref{bruhat-lem} then implies that $sx <y$, and so $m_{sx,y}$ and $n_{sx,y}$ both belong to $v^{-1}\ZZ[v^{-1}]$. If $sx>x$ then it follows by \eqref{thefrm}  that $m_{x,y}$ and $n_{x,y}$ are contained in $v^{-2}\ZZ[v^{-1}]$ so necessarily
 $\mu_m(x,y) = \mu_n(x,y) =0$. It remains only to show that if $sx=x$ then $\mu_n(x,y) = 0$; for this, we note that if $sy<y$ and $sx=x$ then Corollary \ref{M-cor}(a) reduces to the formula  
 \[ \mu_n(x,y) = -\sum_{\substack{x<t<sy \\ st<t}} \mu_n(t,sy) \mu_n(x,t).\] 
We may assume by induction that $\mu_n(x,t) = 0$ 
for all $t \in X$ with $x<t<y$ and $st<t$, 
and so we conclude from this formula that $\mu_n(x,y) = 0$ as desired.
\end{proof}
 
Define  $\tau_m,\tau_n : X \to \cP(S)$
as the maps with
\[ \tau_m(x) = \{ s \in S : sx\leq x\}
\qquand
 \tau_n(x) = \{ s \in S : sx\geq x\}
 \]
and let
 $\omega_m : X \times X \to \ZZ$
 be the 
 map
with
 \[
 \omega_m(x \to y) =\begin{cases}  \mu_m(x,y) + \mu_m(y,x) &\text{if }\tau_m(x) \not\subset \tau_m(y) \\ 0 &\text{if }\tau_m(x)\subset \tau_m(y) .\end{cases}\] 
 Define $\omega_n : X\times X \to \ZZ$ by the same formula, but with  $\mu_m$  and $\tau_m$ replaced by  $\mu_n$ and $\tau_n$. 
 
\begin{theorem}\label{Wgraph-thm}
Both $\Gamma_m = (X,\omega_m,\tau_m)$ and $\Gamma_n = (X,\omega_n,\tau_n)$ are  quasi-admissible $W$-graphs.
\end{theorem}

\begin{proof}
To see that $\Gamma_n$ is a $W$-graph, 
observe that  Lemma \ref{omega-lem}(b) implies that the formula in  Theorem \ref{M-thm}(b) for the action  of $H_s \in \H$ on $\underline N_x \in \cN$ for $s \in S$ and $x \in X$  can be written as
\[ H_s \underline N_x = \begin{cases} v\underline N_x & \text{if $sx<x$} \\ 
-v^{-1} \underline N_x + \ds\sum_{y \in X; \hs sy<y} \Bigr( \mu_n(x,y) + \mu_n(y,x)\Bigl) \underline N_y &\text{if $sx\geq x$}.\end{cases}\]
One checks that this  coincides with the $\H$-module structure described in Definition \ref{Wgraph-def} for the maps $\tau=\tau_n$ and $\omega=\omega_n$.

To prove that $\Gamma_n$ is a $W$-graph, recall the definition of the elements $\underline  N_x' \in \cN$ for $x \in X$ from \eqref{qpcanon2-eq}. By Theorem \ref{M-thm} and Lemmas \ref{barop-lem2},   \ref{prime-lem}, and  \ref{omega-lem}, it holds that  if $s \in S$ and $x \in X$ then
\be\label{above} H_s \underline N_x' = \begin{cases} v \underline N_x' + \ds\sum_{y \in X; \hs sy\leq y}\Bigl( \mu_m(x,y) + \mu_m(y,x)\Bigr) \underline N_y' &\text{if }sx>x \\ -v^{-1} \underline N_x' &\text{if }sx\leq x.\end{cases}
\ee
The matrix of the action of $H_s$ on the basis $\{ \underline N_x'\}_{x \in X} \subset \cN$ is evidently the transpose of the action 
proscribed by Definition \ref{Wgraph-def} with $\tau=\tau_m$ and $\omega=\omega_m$.
Since $\H$ is the quotient of the free $\cA$-algebra generated by $\{ H_s : s \in S\}$ by relations which are invariant under taking transposes, it follows that $\Gamma_m$ is a $W$-graph. 


By Proposition \ref{parity-prop}, the division of $X$ into elements of even and odd height affords a bipartition of $\Gamma_m$ and $\Gamma_n$. Properties (a) and (c) in Definition \ref{admissible-def} hold by construction, so we conclude that $\Gamma_m$ and $\Gamma_n$ are both quasi-admissible.
\end{proof}
\begin{remark} If $(X,\h) = (W,\ell)$ then  $\Gamma_m = \Gamma_n$ and both of these graphs coincide with the original admissible $W$-graph structure on $W$ described in \cite{KL}.
If $W$ is finite and $(X,\h)= (W^J,\ell)$ for some $J\subset S$ as in Example \ref{parabolic-ex}, then $\Gamma_m$ and $\Gamma_n$ are distinct but still admissible, and are isomorphic to the  subgraphs  of the $W$-graph on $W$ induced on the respective vertex sets 
\[
W^{J,\max} = \{ w \in W : ws<w \text{ for all }s \in J\}
\qquand
W^J = \{ w \in W : ws>w\text{ for all }s \in J\} 
.\]
This result does not seem to be well-known, and  originates in work of Couillens \cite{Couillens}; see Chmutov's thesis \cite[\S1.2.4]{Chmutov} for an exposition, as well as the related papers of Howlett and Yin \cite{HY,HY2}.
\end{remark}



In the literature on $W$-graphs, strongly connected components (in a $W$-graph $\Gamma$) are referred to as \emph{cells}. 
as explained at the beginning of this section, the cells of $\Gamma$ define a filtration of its corresponding $\H$-module, and so classifying the cells is a natural problem of interest.
When $(X,\h) = (W,\ell)$ the cells of $\Gamma_m = \Gamma_n$ are the \emph{left cells} of $(W,S)$, about which there exists a substantial literature; see \cite[Chapter 6]{CCG} for an overview.
It is a natural open problem to study to cells of the $W$-graphs $\Gamma_m$ and $\Gamma_n$ defined in this section for more general quasiparabolic sets.

\section{Quasiparabolic conjugacy classes}\label{qcc-sect}

Rains and Vazirani mention two  $W$-actions motivating their study of quasiparabolic sets in \cite{RV}: the action of $W$ on cosets of standard parabolic subgroups, and the action of $W$ on itself by (twisted) conjugation. 
The quasiparabolic $W$-set coming from the former example is relatively well understood, 
having been studied, for example, in \cite{Couillens,Deodhar,Deodhar2,Soergel}.
This section is devoted to  quasiparabolic twisted conjugacy classes, about which less is known.

\subsection{Necessary properties}\label{cc-sect}

 Let $\Aut(W,S)$ denote the group of automorphisms of $W$ which preserve the set of simple generators $S$, and for  $\theta \in \Aut(W,S)$ define sets $W^{\theta,+}$ and $W^+$ by
 \[ W^{\theta,+} = \{ (x,\theta) : x \in W\} = W \times \{\theta\} \qquand W^+=W\times \Aut(W,S) .\]
One gives a group structure to the set $W^+$ via the multiplication formula 
\[ (x,\alpha)(y,\beta) = (x\cdot \alpha(y),\alpha\beta).\]
The group $W^+$   is  a semidirect product $  W \rtimes \Aut(W,S)$, which we sometimes refer to as the \emph{extended (Coxeter) group} of $W$. 
We view $W$ and $\Aut(W,S)$ as subgroups of $W^+$ by identifying $x \in W$ and $\theta \in \Aut(W,S)$ with   $(x,1)$ and $(1,\theta)$, respectively.
The group $W$ acts by conjugation on $W^+$, and for each $\theta \in \Aut(W,S)$ the subset $W^{\theta,+}\subset W^+$ is   a union of $W$-conjugacy classes. The conjugation action of $W$ on $W^{\theta,+}$ coincides with the \emph{$\theta$-twisted conjugation} action of $W$ on itself. 
We identify each ordinary conjugacy class in $W$ with a $W$-conjugacy class in the set $W^{\id,+}\subset W^+$.

Extend the length function on $W$ to $W^+$ by setting 
$\ell(x,\theta) = \ell(x)$.
%
%
%
%
%
%
Any $W$-conjugacy class $\cK$ in $W^+$ is  
then a scaled $W$-set with respect to the height function 
$ \h (w) = \tfrac{1}{2} \ell(w)$ for $w \in \cK$.  
If  this scaled $W$-set is quasiparabolic, then we say that  $\cK$ 
is \emph{quasiparabolic}.

\begin{example}
Consider the $W\times W$-conjugacy class of $(1,\theta) \in (W\times W)^+$, where $\theta \in \Aut(W\times W)$ is the automorphism $\theta :(x,y)\mapsto (y,x)$. This conjugacy class (with the height function $\tfrac{1}{2}\ell$) is isomorphic as a scaled $W$-set to the quasiparabolic set $(W,\ell)$ from Example \ref{case0-ex}. 
Via this example, one can view our results concerning quasiparabolic conjugacy classes  as generalizing constructions (e.g., the Kazhdan-Lusztig basis) attached to $W$ itself.
\end{example}

The main object of this section is to say something about when a $W$-conjugacy class in $W^+$ is quasiparabolic.
We will need the following lemma, which is   similar to a property Rains and Vazirani check in the course of their proof of \cite[Theorem 3.1]{RV}.

\begin{lemma} Let $w \in W^+$ and $ r\in R $ and $s \in S$. 
\ben
\item[(a)] If $\ell(wr) > \ell(w)$ and $\ell(swr) < \ell(sw)$ then $swr=w$.
\item[(b)] If $\ell(rw) > \ell(w)$ and $\ell(rws) < \ell(ws)$ then $rws = w$.
\een
\end{lemma}

\begin{proof} 
We  only prove part (a) since the other part is equivalent via the identity $\ell(x) = \ell(x^{-1})$.
Since $R$ is preserved by every $\theta \in \Aut(W,S)$,
to prove part (a)  for all $w \in W^+$ it suffices to check the given statement for $w \in W$.
Proceeding, suppose $w \in W$ is such that 
$\ell(wr) > \ell(w)$ and $\ell(swr) < \ell(sw)$. 
Let $w = s_1s_2\cdots s_k$ be a reduced expression; then 
 $sw=ss_1s_2\cdots s_k$ is also a reduced expression since $\ell(sw) = \ell(w)+1$ as $\ell(sw)>\ell(swr) \geq \ell(wr)-1 \geq \ell(w)$.
 Given that $\ell(swr)< \ell(sw)$, the  Strong Exchange Condition \cite[Theorem 5.8]{Hu}
 implies that either $swr = w$ or $swr = ss_1\cdots s_{i-1} s_{i+1}\cdots s_k$ for some $1 \leq i \leq k$. The latter case cannot occur, since it implies that $wr = s_1\cdots s_{i-1}s_{i+1}\cdots s_k$ which in turn implies the contradiction $\ell(wr) \leq k-1 < \ell(w)$. 
\end{proof}

Define $\Des(w) = \{ s \in S: \ell(sw) < \ell(w)\}$ for $w \in W^+$.

\begin{theorem}\label{qp-big-lem}
Fix $\theta \in \Aut(W,S)$ and let $\cK \subset W^{\theta,+}$ be a quasiparabolic $W$-conjugacy class.
Suppose $w=(x,\theta) \in W^+$ is the unique $W$-minimal element of $\cK$ and define $J =\Des(w)$. 
\ben

\item[(a)] For all $s \in J$ it holds that $sws= w$.

\item[(b)] The standard parabolic subgroup $W_J\subset W$ is finite and preserved by $\theta$.

\item[(c)] It holds that $x=w_J$ where $w_J$ denotes the longest element in $W_J$.
\een
\end{theorem}

\begin{proof}
If $x=1$ then $J=\varnothing$ and parts (a)-(c) hold vacuously. Therefore assume $\ell(x) = \ell(w) > 0$.
To prove part (a), note that if $s \in J$ then  we have $\ell(w) \leq \ell(sws) \leq \ell(sw)+1 = \ell(w)$ since $w$ is minimal in its conjugacy class, so $\ell(sws)=\ell(w)$ which implies that $sws=w$ since the conjugacy class of $w$ is quasiparabolic.

For the first assertion in part (b), observe that $x^{-1} wx = (\theta(x),\theta)$ is $W$-conjugate to $w$ and has the same length, so since $w$ is the unique minimal element in its conjugacy class we must have $x=\theta(x)$, which implies that $J = \theta(J)$.

 Fix $k\geq 1$ and let $s_i \in J$ be such that $s_1s_2\cdots s_k$ is a reduced expression. Define $w_0=w$ and $w_i = w_{i-1}s_i = ws_1s_2\cdots s_i$ for $i\geq 1$. We   claim that $\ell(w_i) = \ell(w) - i$ for all $0 \leq i \leq k$.
We prove this by induction on $i$; the claim is  true if $i \in \{0,1\}$ by part (a), so assume $i\geq 2$ and that $\ell(w_j) = \ell(w)-j$ when $j<i$. By part (a),
 $s_1 w_{i-1} = w s_2 \cdots s_{i-1}$ and $s_1 w_{i-1} s_i = w s_2 \cdots s_{i-1} s_i$. Since $s_2\cdots s_{i-1}$ and $s_2 \cdots s_{i-1}s_i$ are   reduced expressions of length less than $i$, it follows 
 that
\[ 
\ell(s_1 w_{i-1}) = \ell(w_{i-2}) >  \ell(w_{i-1})
\qquand 
\ell(s_1 w_{i-1} s_i) 
= \ell(w_{i-1})
\]
by our inductive hypothesis.
Now observe that if $\ell(w_i) = \ell(w_{i-1}s_i) \neq \ell(w_{i-1}) -1 $ then $\ell(w_{i-1} s_i) =\ell(w_{i-1})+1 >\ell(s_1 w_{i-1} s_i)$,
 in which case   the preceding lemma (with $r=s_1$ and $s=s_i$) gives $s_1 w_{i-1}  = w_{i-1}s_i$ which  implies  that $s_2\cdots s_{i-1} = s_1s_2 \cdots s_{i-1}s_i$. The last identity contradicts our assumption that $s_1\cdots s_k$ is a reduced expression,
 so we must have $\ell(w_i) = \ell(w_{i-1}) -1 = \ell(w) - i$ as desired. Our claim thus holds for all $i$ by induction.

It follows from the claim just proved that if $z \in W_J$ then $\ell(wz) = \ell(w) - \ell(z)$. 
Since the length of $wz$ is necessarily nonnegative, we deduce that $W_J$ must be finite, which completes the proof of part (b). To prove part (c), let $s_i \in S$ be such that $x=s_1\cdots s_k$ is a reduced expression.
Since $w_J  = \theta(w_J)$ by part (b), our claim implies that $\ell(xw_J ) = \ell(x\theta(w_J))=\ell(ww_J) = \ell(x)-\ell(w_J)$.
We may therefore assume that for some $j \geq 1$ it holds that $s_js_{j+1}\cdots s_k$ is a reduced expression for $w_J^{-1} = w_J$. We now argue that $j=1$.
To show this, observe that  $s_1 \in J=\Des(w_J)$, so
by our claim and part (a) it follows that  
\[  \ell( s_1 w w_J) = \ell(ws_1w_J) = \ell(w) - \ell(s_1w_J) > \ell(w)-\ell(w_J)=\ell(ww_J).\]
Thus $s_1 \notin \Des(ww_J)$, which clearly only holds if $j=1$, since $ww_J$ has length $j-1$ and if $j>1$ then  $ww_J = (s_1\cdots s_{j-1},\theta)$. We conclude that $x=w_J$ which  proves part (c).
\end{proof}

Given $w \in W^+$ and $H \subset W$ and $\theta \in \Aut(W)$, define the following subgroups:
\[ C_W(w) = \{ x \in W: xw=wx\}\qquand N_{W,\theta}(H) = \{ x \in W : xH=H\theta(x)\}.\]
The first subgroup is the usual centralizer  while the second is a twisted normalizer. 

\begin{corollary} If $w=(x,\theta) \in W^+$ is the unique $W$-minimal element in quasiparabolic $W$-conjugacy class then $ C_{W}(w) = N_{W,\theta}(W_J)$ where $J = \Des(w)$.
\end{corollary}

\begin{proof}
Theorem \ref{qp-big-lem} shows that $x$ is both central and equal to the longest element $w_J$ in $W_J$. Pfeiffer and R\"ohrle have shown that usual centralizer $C_W(w_J)$ is equal to the usual normalizer of $W_J$ if $w_J$ is central in $W_J$ \cite[Proposition 2.2]{PR}; their proof of this fact carries over to our slightly more general twisted situation with almost no modification.
\end{proof}

We state below three more corollaries, after introducing some more  notation.
First define 
\[  \I = \I(W,S) \omdef= \{ w \in W^+ : w^2=1\}.\]
We refer to  elements of  $\I$ as \emph{twisted involutions}. Observe that a pair $(x,\theta) \in W^+$ is a twisted involution if and only if $\theta^2 =1$ and 
 $\theta(x) = x^{-1}$; in this situation;  the element $x \in W$ is sometimes referred to as a twisted involution relative to $\theta$.
 Additionally, for $\theta \in \Aut(W,S)$ define 
 \[ \iota(\theta) \omdef= \left\{ \( x^{-1}  \theta(x), \theta\) \in W^+ : x \in W\right\}\]
 Observe that $\iota(\theta) $ is the $W$-conjugacy class of $(1,\theta) \in W^+$, so  $\iota(\id) = \{1\} \subset W$ and if $\theta^2=1$ then $\iota(\theta) \subset \I$.
When $\theta^2=1$,
Hultman \cite{H3} refers to the elements of $\iota(\theta)$ as \emph{twisted identities}.
 Both $\I$ and $\iota(\theta)$ have a number of interesting properties; see, for example, \cite{H1,H2,H3,R,RS,S}.

\begin{corollary}\label{big-cor1} Let $\theta \in \Aut(W,S)$ and let $\cK \subset W^{\theta,+}$ be a quasiparabolic $W$-conjugacy class. The operation $w\mapsto w^2$ then defines a surjective map $\cK \to \iota(\theta^2)$.
\end{corollary}

\begin{proof}
Let $w=(x,\theta)$ be the unique minimal element in $\cK$. By Theorem \ref{qp-big-lem}, $x=w_J$ for a $\theta$-invariant subset $J \subset S$, so $x=x^{-1} = \theta(x)$ and 
$w^2 = (1,\theta^2) \in \iota(\theta^2)$, so the corollary follows.
\end{proof}

\begin{corollary}\label{big-cor2} If $\theta^2=1$ then all quasiparabolic $W$-conjugacy classes in $W^{\theta,+}$ are subsets of $\I$.
\end{corollary}

\begin{proof}
This follows from the preceding corollary since  if $\theta^2=1$ then the preimage of $\iota(\theta^2)= \{1\}$ under $w\mapsto w^2$ is precisely $\I$.
\end{proof}


\begin{corollary}\label{big-cor3} All quasiparabolic conjugacy classes in $W$ are subsets of $ \{w \in W : w^2=1\}$.
\end{corollary}

\begin{proof}
This follows from the previous corollary since $\{w \in W : w^2=1\} = W^{\id,+} \cap \I$.
\end{proof}

\begin{remark}
When $W$ is finite, this last corollary follows more directly from the well-known fact (discussed, for example, in the introduction of \cite{LSV}) that every element of $ W$ is conjugate to its inverse, so only conjugacy classes of involutions have unique minimal elements. In an infinite Coxeter group an element can fail to be conjugate to its inverse. 
\end{remark}

Our last result in this section is the following theorem promised in the introduction. 

\begin{theorem}\label{qp-finite-thm} If $W$ is finite, then all quasiparabolic conjugacy classes in $W^+$ are subsets of $\I$.
\end{theorem}

We prove the theorem after stating two   preliminary lemmas.
Recall that a Coxeter system $(W,S)$ is \emph{irreducible} if no proper nonempty subset $J \subset S$ is such that $st=ts$ for all $s \in J$ and $t \in S \setminus J$. If $J \subset S$ then we write $W_J$ for the subgroup which $J$ generates; then $(W_J,J)$ it itself a Coxeter system, whose length function coincides with the restriction of $\ell :W \to \NN$.
Define 
\[\cJ = \cJ(W,S) = \{  J : \varnothing \subsetneq J \subset S \text{ such that }(W_J,J)\text{ is irreducible}\}.\]
For each $J \in \cJ$ we denote by $\pi_J : W\to W_J$ 
the unique surjective homomorphism with $\pi_J(s) = s$ for $s \in J$ and $\pi_J(s) = 1$ for $s \in S\setminus J$.
The map \[w \mapsto \(\pi_J(w)\)_{J \in \cJ}\] is then an isomorphism of Coxeter systems $W \xrightarrow{\sim} \prod_{J \in \cJ} W_J$, where the product group is interpreted as a Coxeter system relative to the generating set
given by the image of $S$.

Fix $\theta\in \Aut(W,S)$ and  note that $\theta$ permutes the set $\cJ$, in the sense that $\theta(J) \in \cJ$ for all $J \in \cJ$.
Given $J \in \cJ$, let $J_1,J_2,\dots,J_k$ be the distinct elements of the $\langle \theta \rangle$-orbit of $J$, ordered such that $J = J_1$ and $\theta(J_i) = J_{i+1}$ (indices interpreted modulo $k$).
Define $W_{J,\theta} = W_{J_1} \times W_{J_2} \times \cdots \times W_{J_k}$ and let $\tau_\theta$ be the automorphism of $W_{J,\theta}$ with 
\[ \tau_\theta (x_1,\dots,x_{k-1},x_k) = (\theta(x_k),\theta(x_1),\dots,\theta(x_{k-1}))\qquad\text{for }x_i \in W_{J_i}.\]
Note that $(W_{J,\theta},K)$ is a Coxeter system when 
$K$ is the smallest set preserved by $\tau_\theta$ which contains
 $(s,1,\dots,1 ) \in W_{J,\theta}$ for all $s \in J$, and $\tau_\theta \in \Aut(W_{J,\theta},K)$.
Define $ \pi_{J,\theta} :W^{\theta,+} \to (W_{J,\theta})^{\tau_\theta,+}$
by
\[ \pi_{J,\theta}(x,\theta) = \( \( \pi_{J_1}(x), \pi_{J_2}(x),\dots, \pi_{J_k}(x)\), \tau_\theta\)\quad\text{for $x \in W.$}\]
We now state two lemmas using this formalism.

\begin{lemma} \label{I-lem1} Fix $\theta \in \Aut(W,S)$ and let $\cK \subset W^{\theta,+}$ be a $W$-conjugacy class. 
\ben
\item[(a)] For each $J \in \cJ(W,S)$, the image $\pi_{J,\theta}(\cK)$
is  a $W_{J,\theta}$-conjugacy class.
\item[(b)] $\cK$ is quasiparabolic if and only if $\pi_{J,\theta}(\cK)$ is quasiparabolic for every $J \in \cJ(W,S)$.
\een
\end{lemma}

\begin{proof}
We just sketch the idea of a proof of this result, which is intuitively clear. Part (a) follows by elementary considerations. The ``only if'' direction of part (b) follows from \cite[Proposition 2.6]{RV} (which states that a set which is quasiparabolic relative to the action of a Coxeter group is also quasiparabolic relative to any of the group's standard parabolic subgroups) while the ``if'' direction follows from \cite[Proposition 3.3]{RV} (which states that the Cartesian product of several quasiparabolic sets is a quasiparabolic set relative to the Cartesian product of the acting Coxeter groups). 
\end{proof}

\begin{lemma} \label{I-lem2}
Suppose $\theta \in \Aut(W,S)$  transitively permutes $\cJ=\cJ(W,S)$. Assume $|\cJ| \geq 2$ and let $\cK \subset W^{\theta,+}$ be a $W$-conjugacy class. 
\ben
\item[(a)] If $|\cJ|>2$ then $\cK$ is not quasiparabolic.

\item[(b)] If $|\cJ|=2$  then $\cK$ is quasiparabolic if and only if its minimal element is $(1,\theta) \in W^+$.
\een
Hence, if $\cK$ is quasiparabolic then  $ \cK\subset \I$.
\end{lemma}

\begin{proof}
Let $k = |\cJ(W,S)|$.
Since $\theta$ transitively permutes the   elements of $\cJ(W,S)$,  we can assume without loss of generality that  $W = W'\times W' \times \cdots \times W'$ ($k$ factors)  for some Coxeter system $(W',S')$ and   that $\theta$ acts on $W$ by the formula $(w_1,\dots,w_{k-1},w_k) \mapsto (w_k,w_1,\dots,w_{k-1})$. 

Suppose   $\cK$ is  quasiparabolic, and 
let  $w_i \in W'$ be the elements such that $w=\((w_1,\dots,w_k),\theta\) \in \cK$ is the unique element of minimal length.
We then must have $w_1=\dots=w_k=1$, since if some $w_i \neq 1$ then there would exist $s \in S'$ with $sw_i < w_i$, and in this case one can check that if $t=(1,\dots,s,\dots,1) \in S$ is the simple generator with 1 in all but the $i$th coordinate, then $twt \neq w$ has $\ell(twt) \leq \ell(w)$, contradicting Lemma \ref{minimal-lem}.
Hence $\cK$ must contain the element $(1,\theta)$, which  is automatically minimal since it has length 0.

We now argue that the case $k\geq 3$ leads to contradiction. For this, choose any $r \in S'$, and define $s,t \in S$ by  $s = (r,1,1,\dots) $ and $t = (1,r,1,\dots)$. If $k\geq 3$ then the element $x = s(1,\theta)s  = \( (s,s,1,\dots,),\theta\) \in \cK$
  has $txt = \( (s,1,s,\dots),\theta\)\neq x $ but $\h(txt) = \h(x) = 1$. This  contradicts (QP1) in the definition of a quasiparabolic set, so we conclude that $k=2$, which proves part (a) and one direction of part (b).
For the rest of part (b),
it remains to check that the $W$-conjugacy class of $(1,\theta)$ is in fact quasiparabolic when $k=2$.
This follows as a standard exercise from properties of the Bruhat order of $W$; 
alternatively, the desired claim is a   consequence of  a general criterion of Rains and Vazirani which we will restate below as Theorem \ref{perfect-thm}.
\end{proof}

Finally, we prove Theorem \ref{qp-finite-thm}.

\begin{proof}[Proof of Theorem \ref{qp-finite-thm}]
Let $\theta \in \Aut(W,S)$ and suppose $\cK \subset W^{\theta,+}$ is a quasiparabolic $W$-conjugacy class. To show that $\cK \subset \I$, 
we reduce via Lemma \ref{I-lem1}  to the case when $\theta$ acts transitively on $\cJ(W,S)$.
In this situation, Lemma \ref{I-lem2} implies that if $\cJ(W,S)$ has $k\geq 2$ elements then $\cK \subset \I$ (and in fact $k=2$). On the other hand,
if $\theta^2=1$ then by Corollary \ref{big-cor2} we likewise have $\cK \subset \I$.
It thus only remains to show that $\cK \subset \I$ if $(W,S)$ is irreducible and $\theta^2\neq 1$. 
This actually leaves very little left to check: for if $(W,S)$ is finite and irreducible and $\theta \in \Aut(W,S)$ is not an involution, then by the classification results in \cite[Chapter 2]{Hu}, $(W,S)$ is necessarily of type $D_4$ and $\theta$ can be identified with the automorphism of order three described in \cite[\S6.2]{GKP}.

Explicitly, let $W$ be the Coxeter group of $D_4$, i.e. the group generated by the set of involutions $S = \{s_1,s_2,s_3,s_4\}$,
where $s_1,s_2,s_4$ pairwise commute and $s_is_3$ has order 3 for $i \in \{1,2,4\}$.
Assume $\theta \in \Aut(W,S)$ is not an involution.
 Then, after possibly relabeling the simple generators, we may assume that $\theta$ acts on $S$ by mapping $s_1 \mapsto s_3$ and $s_3 \mapsto s_4$  and $s_4 \mapsto s_1$ and $s_2\mapsto s_2$.   Calculations of Geck, Kim, and Pfeiffer (see \cite[Table I]{GKP})   show that only two $W$-conjugacy classes in $W^{\theta,+}$ have unique elements of minimal length, namely, the conjugacy classes of $(1,\theta)$ and $(s_2,\theta)$. One checks that both classes violate (QP1) in Definition \ref{qp-def}: the first class contains $x=s_1(1,\theta)s_1 = (s_1s_3,\theta)$ which has the same length as $s_3xs_3 \neq x$, while the second class contains $y = s_1s_2s_3(s_2,\theta)s_3s_2s_1$ which has the same length as $s_2ys_2 \neq y$.
We conclude that   $\cK\subset \I$, as desired.
\end{proof}

\subsection{Sufficient conditions}\label{suff-sect}

Rains and Vazirani prove a useful sufficient condition for a $W$-conjugacy class of twisted involutions to be quasiparabolic.
Recall that $R = \{wsw^{-1} : (w,s) \in W \times S\}$ is the set of reflections in $W$.
\begin{definition} A twisted involution $w \in \I$ is \emph{perfect} if $(rw)^4 = 1$ for all $r \in R$. 
\end{definition}
Observe that if $w \in \I$ is perfect then
all elements in the $W$-conjugacy class of $w$ are also perfect, so it makes sense to say that a $W$-conjugacy class of twisted involutions is \emph{perfect} if any of its elements are.
The following appears as \cite[Theorem 4.6]{RV}.

\begin{theorem}[Rains and Vazirani \cite{RV}] \label{perfect-thm} All perfect conjugacy classes in $\I $  are quasiparabolic.
\end{theorem}


As Rains and Vazirani note in \cite{RV}, it is straightforward to check that all fixed-point-free involutions in $S_{2n}$ are perfect. Therefore: 

\begin{corollary}[Rains and Vazirani \cite{RV}] \label{fpf-cor} The conjugacy class of fixed-point-free involutions in the symmetric group $S_{2n}$ is quasiparabolic for all $n$.
\end{corollary}

Rains and Vazirani describe explicitly the perfect $W$-conjugacy classes in $W^+$ when $W$ is finite in \cite[Example 9.2]{RV}.
There can exist quasiparabolic conjugacy classes in $\I$ which are not perfect, however, even when $W$ is finite. For example: 
\begin{itemize}
\item If $(W,S)$ has type $F_4$, then we have checked with a computer that the conjugacy class of the nontrivial diagram automorphism in $\Aut(W,S)\subset \I$  has 72 elements and is quasiparabolic but not perfect. 

\item If $(W,S)$ has type $I_2(2m)$,  then the conjugacy classes of each simple generator are disjoint  of size $m$, while 
the conjugacy class of the nontrivial diagram automorphism in $\Aut(W,S) \subset \I$ has size $2m$. All three conjugacy classes are quasiparabolic, but the first two are perfect only when $m \in \{1,2\}$ while the third is   perfect only when $m=1$.

\end{itemize}
By appealing to Theorem \ref{qp-finite-thm} and using a computer for the exceptional types, one can show that when $(W,S)$ is an irreducible finite Coxeter system   these are the only examples of quasiparabolic $W$-conjugacy classes in $W^+$ which are not perfect. 
Combining this  with Lemmas \ref{I-lem1} and \ref{I-lem2} and the discussion in \cite[Example 9.2]{RV} would afford a  classification 
of all quasiparabolic conjugacy classes in a finite (extended) Coxeter group.
We do not pursue this topic here, however.


We can describe examples of quasiparabolic conjugacy classes which are not comprised of twisted involutions.
A Coxeter system $(W,S)$ is \emph{universal} if $st$ has infinite order for all distinct generators $s,t \in S$. 
Each element of a universal Coxeter group has a unique reduced expression. 

\begin{proposition} \label{universal-qp-thm} Suppose $(W,S)$ is a universal Coxeter system.
Let $\cK \subset W^+$ be a $W$-conjugacy class. The following are then equivalent:
 \ben
 \item[(a)] $\cK$ is quasiparabolic.
 \item[(b)] $\cK$ contains a unique minimal element.
 \item[(c)] $\cK$ contains an element $(x,\theta) \in W^+$ with $x=\theta(x)$ and $x \in\{1\}\cup S$.
 \een
\end{proposition}

\begin{remark}
Note in the situation of (c) that $(x,\theta)$ has length $0$ or $1$ and so is necessarily an element of minimal length in $\cK$, as conjugation preserves length parity.
\end{remark}

\begin{proof}
By Lemma \ref{minimal-lem}, (a) $\Rightarrow$ (b) so we only need to show that (c) $\Rightarrow$ (b) and (c) $\Rightarrow$ (a).
For the first implication,
suppose $w= (x,\theta) \in W^+$ is the unique minimal element in its $W$-conjugacy class. 
Since the conjugate element $x^{-1} wx = (\theta(x),\theta)$ has the same length as $w$, we must have $x=\theta(x)$. We wish to show that  $x \in \{1\} \cup S$. 
If $x\neq 1$ then there is a unique reduced expression $x=s_1s_2\cdots s_k$ where $k\geq 1$. The conjugate element $s_1 ws_1 = (s_2\cdots s_k \theta(s_1),\theta)$ then has   length $\ell(w)$ or $\ell(w)-2$; 
since $w$ is the unique minimal element in its conjugacy class, the latter case cannot occur and we must have $s_1s_2\cdots s_k = s_2\cdots s_k \theta(s_1)$. Both of these expressions are reduced, so they can be equal only if $k=1$, in which case $x \in S$. 

This shows that (b) $\Rightarrow$ (c) and it remains only to show that (c) $\Rightarrow$ (a). For this, suppose $w=(x,\theta) \in W^+$ such that $x=\theta(x) \in \{1\} \cup S$, so that $\ell(w) \in \{0,1\}$. 
Since $W$ is universal, the centralizer $C_W(w) = \{ z \in  W : zw=wz\}$ is   given by
\[C_W(w) = W_J
\qquad\text{where}\qquad J = \begin{cases} \{x\} &\text{if }x \in S \\ \{ s \in S : \theta(s) = s\}&\text{if }x=1.\end{cases}\]
It follows by \cite[Proposition 2.15]{RV} that the $W$-conjugacy class of $w$ is isomorphic as a scaled $W$-set (after translating the height function by $\frac{1}{2} \ell(w)$) to $(W^J,\ell)$. Since the latter set is quasiparabolic, so is the former, and thus (c) $\Rightarrow$ (a) as required.
\end{proof}

\begin{corollary} Suppose $(W,S)$ is a universal Coxeter system. Then all $W$-conjugacy classes in $\I$ are quasiparabolic, but there exist quasiparabolic $W$-conjugacy classes in $W^+ $ which are not contained in $ \I$ whenever $|S|\geq 3$.
\end{corollary}

\begin{proof}
Let $\cK \subset \I$ be a $W$-conjugacy class and let $w=(x,\theta) \in \cK$ be some minimal element. To show that $\cK$ is quasiparabolic it suffices by the proposition
to show that $w$ is the unique minimal element in $\cK$.
If $x=1$ then this is clear, so suppose $x\neq 1$ and choose $s \in S$ such that  $\ell(sw) < \ell(w)$ for some $s \in S$. The minimality of $w$ implies $\ell(sws) = \ell(w)$, which implies $sw=ws$ by a straightforward argument using the (weak) Exchange Condition; see \cite[Lemma 3.4]{H2}. The identity $sw=ws$ is equivalent to $sx = x\theta(s)$, which can hold only if $x=s=\theta(s)$ since $W$ is universal. By the proposition we therefore conclude that $w$ is  the unique minimal element in $\cK$ as desired.

Proposition \ref{universal-qp-thm} shows that when $W$ is universal the $W$-conjugacy class of $(1,\theta)$ is quasiparabolic for any $\theta \in \Aut(W,S)$. This conjugacy class is not a subset of $\I$ whenever $\theta^2 \neq 1$, which can occur if $|S| \geq 3$ since $\Aut(W,S)$ is isomorphic to the group of permutations of $S$.
\end{proof}

As noted in the proof of Proposition \ref{universal-qp-thm}, if $(W,S)$ is universal and $\cK\subset W^+$ is a quasiparabolic  $W$-conjugacy class, then $(\cK,\frac{1}{2}\ell)$ may be identified (after possibly translating the height function) with the quasiparabolic set ($W^J,\ell)$ for some $J \subset S$. Thus, Corollary \ref{freebar-cor} implies the following:

\begin{corollary} If $(W,S)$ is a universal Coxeter system, then all quasiparabolic $W$-conjugacy classes in $W^+$ admit bar operators. 
\end{corollary}


\subsection{Bar operators for  twisted involutions}\label{barop-concrete-sect}

Let $(W,S)$ be any  Coxeter system and write $\I_\QP = \I_\QP(W,S)$ for the union of all quasiparabolic $W$-conjugacy classes in $\I \subset W^+$. This union is, by construction, a quasiparabolic set relative to the height function $\frac{1}{2}\ell$, and it is bounded below.
 Given $w \in W^+$,  define
\[ |w|_m =v^{\ell_{\min}(w)}
\qquand
 |w|_n =(-v)^{-\ell_{\min}(w)} = (-1)^{\ell(w)} / |w|_m\]
 where $\ell_{\min}(w) = \min_{x \in W} \ell(xwx^{-1})$. Note that these quantities depend only on the $W$-conjugacy class of $w$.
Our main result in this section is the following theorem from the introduction.

 \begin{theorem}\label{Ibarop-thm}
The quasiparabolic set $(\I_\QP,\tfrac{1}{2}\ell)$ admits a bar operator. 
The (unique) bar operators on the corresponding $\H$-modules  $\cM = \cM(\I_\QP,\frac{1}{2}\ell)$ and $\cN = \cN(\I_\QP,\frac{1}{2}\ell)$ act by the formulas
\[ \overline{M_{(x,\theta)}}  =|(x,\theta)|_m\cdot  \overline{ H_{x}} \cdot M_{{(x^{-1},\theta)}}
\quand
 \overline{N_{(x,\theta)}}  =|(x,\theta)|_n\cdot  \overline{ H_{x}} \cdot N_{{(x^{-1},\theta)}}
\qquad\text{for $(x,\theta) \in \I_\QP$.}\]
 \end{theorem}

 \begin{remark}\label{lvremark}
 In \cite{LV2,LV1,LV3}, Lusztig and Vogan study a module of the Iwahori-Hecke algebra of an arbitrary Coxeter system on the free $\cA$-module generated by all of $\I$.
They prove (see \cite[Theorems 0.2 and 0.4]{LV2}) that this module possesses a  ``bar operator'' admitting a unique invariant ``canonical basis.'' 
The formula for their bar operator is the same (up to scaling factors) as the ones given in the preceding theorem, but there does not appear to be any simple relationship between Lusztig and Vogan's $\I$-indexed canonical basis and the two canonical bases we obtain for  $\cM(\I_\QP,\frac{1}{2}\ell)$ and $ \cN(\I_\QP,\frac{1}{2}\ell)$; see Problem \ref{lastprob}. Indeed, the polynomial coefficients of these bases satisfy different degree bounds, and the two $W$-graph structures on $\I_\QP$ afforded by Theorems \ref{Wgraph-thm} and \ref{Ibarop-thm} can each have cells which are not subsets of the two-sided cells in $W$ (as defined in \cite{KL}). By contrast, the results of \cite[Section 5]{LV1} show that the cells of the natural ``$W$-graph'' structure on $\I$ induced by Lusztig and Vogan's canonical basis are always contained in a two-sided cell. 
 \end{remark}

\begin{proof}
We prove the statement for the module $\cM$; the proof for $\cN$ is very similar.
First,
we check that 
$\overline{M_w} = M_w$ if $w=(x,\theta) \in \I_{\QP}$ is the unique element of minimal length in its $W$-conjugacy class.
By Theorem \ref{qp-big-lem} there exists a $\theta$-invariant subset $J \subset S$ such that $x=x^{-1}= w_J$ and $sws=w$ for all $s \in J$. In any reduced expression $x=s_1s_2\cdots s_{k}$ every factor satisfies $s_i \in J$ (see \cite[Theorem 5.5]{Hu}), and so 
\[ \overline{H_x} M_w = H_{s_1}^{-1} H_{s_2}^{-1} \cdots H_{s_{k}}^{-1} M_w  = v^{-\ell(x)} M_w.\] As $\ell(x) = \ell_{\min}(w)$ since $w=(x,\theta)$ is $W$-minimal, it follows that $\overline{M_w} = M_w$ as desired.

According to Definition \ref{barop-def}, 
it now remains only to show that $\overline{HM} = \overline H\cdot  \overline{M}$ for all $ H \in \H$ and $M \in \cM$. For this it suffices to check that 
\[\overline{H_s M_{(x,\theta)}} =\overline {H_s}\cdot \overline{M_{(x,\theta)}}\qquad\text{for $s \in S$ and $(x,\theta) \in \I_\QP$.}\]
Set 
$x'  = sx\theta(s)$
so  that $s (x,\theta) s = (x',\theta),$
and observe that $sx<x$ if and only if $\theta(s) x^{-1} < x^{-1}$ since $(x,\theta) \in \I$ implies $x^{-1} = \theta(x)$.
As an abbreviation we define $\kappa = |(x,\theta)|_m = |(x',\theta)|_m$. There are now three cases to consider, according to the difference in length between $x$ and $x'$:
\ben
\item[(1)] If $\ell(x') > \ell(x)$ then  $H_sH_x = H_{x'}\overline{H_{\theta(s)}}$ and $H_{\theta(s)} M_{(x^{-1},\theta)} = M_{({x'}^{-1},\theta)}$ so 
\[ \overline{H_s} \cdot \overline{M_{(x,\theta)}} = 
\kappa \cdot  \overline {H_{x'}} \cdot H_{\theta(s)} \cdot M_{(x^{-1},\theta)}
= 
\kappa \cdot \overline{H_{x'}} \cdot M_{({x'}^{-1},\theta)} =
 \overline{M_{(x',\theta)}}
 = \overline{H_sM_{(x,\theta)}}.
\]

\item[(2)] If $\ell(x')  < \ell(x)$ then 
$H_s H_x = H_{sx} + (v-v^{-1})H_x=H_{x'}  H_{\theta(s)} + (v-v^{-1})H_x$  so
\[\overline{H_s} \cdot \overline{M_{(x,\theta)}}
=
\kappa \cdot \overline {H_{x'}} \cdot \overline {H_{\theta(s)}} \cdot M_{(x^{-1},\theta)} + (v^{-1}-v) \cdot \overline{M_{(x,\theta)}}.
\]
Since 
$\overline{H_{\theta(s)}} M_{(x^{-1},\theta)} = M_{(x',\theta)} $,
the right side of the preceding identity is equal to
 \[
\overline{ M_{(x',\theta)}} + (v^{-1}-v) \cdot \overline{M_{(x,\theta)}}
=
\overline{H_s M_{(x,\theta)}}.\]

\item[(3)] If $\ell(x')=\ell(x)$ then $x'=x$ by condition (QP1) in Definition \ref{qp-def}, so we have $H_sH_x = H_xH_{\theta(s)}$ and $H_sM_{(x,\theta)} = v M_{(x,\theta)}$, and therefore
\[\overline{H_s} \cdot \overline{M_{(x,\theta)}}
=
\kappa \cdot  \overline {H_{s}}  \cdot \overline{H_x}\cdot M_{(x^{-1},\theta)} =
\kappa \cdot  \overline {H_{x}}  \cdot \overline{H_s}\cdot M_{(x^{-1},\theta)} = v^{-1}  \overline{ M_{(x,\theta)}} = \overline{H_s M_{(x,\theta)}}.
\]
%
%
%
\een
Hence the given $\cA$-antilinear map $\cM\to \cM$ is a bar operator, which is what we set out to prove. 
\end{proof}

Assume $(W,S)$ is a finite Coxeter system, so that $W$ has 
a longest element $w_0$. Recall since the longest element is unique, we have $w_0 = w_0^{-1} =\theta(w_0)$ for all $\theta \in \Aut(W,S)$. Write $\theta_0$ for  the inner automorphism of $W$ given by 
\[ \theta_0 : w \mapsto w_0 ww_0.\] This map is an automorphism of the poset $(W,\leq )$ 
and in particular is length-preserving \cite[Proposition 2.3.4(ii)]{CCG}; thus it belongs to $\Aut(W,S)$. 
In fact, $\theta_0$ lies in the center of $\Aut(W,S)$.
Let 
\[ w_0^+= (w_0,\theta_0) \in W^+.\]
Observe that this element is a central involution in $W^+$,
and so if
 $w = (x,\theta) \in \I$ then $ww_0^+ = w_0^+w = (xw_0,\theta\theta_0) \in \I$.
Relative to this notation, we have the following lemma.

\begin{lemma}\label{w0-lem} 
The map 
$ w \mapsto ww_0^+$ 
defines a Bruhat order-reversing involution of $W^+$ (also, of $\I$) which induces an involution of the set of quasiparabolic conjugacy classes in $W^+$.
\end{lemma}

\begin{proof}
The map $w\mapsto ww_0^+$ is an involution of $W^+$ which preserves $\I$ since $w_0^+$ is a central involution; the map reverses the Bruhat order on $W^+$ by \cite[Proposition 2.3.4(i)]{CCG}. Let $\cK \subset W^+$ be a quasiparabolic $W$-conjugacy class. Since $w_0^+$ is central the set $\cK w_0^+ = \{ ww_0^+ : w \in \cK\}$ is then  also a $W$-conjugacy class  and it remains only to show that it is quasiparabolic.
This is straightforward from Definition \ref{qp-def} since $\cK$ is quasiparabolic and since for any $x \in W$ we have $\ell(xww_0^+x^{-1}) = \ell(xwx^{-1} w_0^+) = \ell(w_0) - \ell(xwx^{-1})$ by \cite[Proposition 2.3.2(ii)]{CCG}.
\end{proof}

Let $\cM=\cM(\I_\QP,\tfrac{1}{2}\ell) $ and $ \cN = \cN(\I_\QP,\tfrac{1}{2}\ell)$ as in Theorem \ref{Ibarop-thm}. For the rest of this section $m_{x,y}$ and $n_{x,y}$ for $x,y \in \I_\QP$ denote the polynomials defined from the canonical bases of these particular modules as in \eqref{qpklpol}. 
When $W$ is finite, can can prove an inversion formula for these polynomials, analogous to \cite[Theorem 3.1]{KL} concerning the original Kazhdan-Lusztig polynomials.

We introduce some notation which will be helpful in  proving this result. Let $\cM^*$ be the $\cA$-module of $\cA$-linear maps $\cM \to \cA$. For   $w \in \I_\QP$ define $M_w^* \in \cM^*$ as the $\cA$-linear map with 
\[M_w^*(M_{w'}) = \delta_{w,w'}\qquad\text{for $w' \in \I_\QP$.}\]
When $W$ is finite, the set of elements $\{ M_w^* : w \in \I_\QP\}$ forms an $\cA$-basis for $\cM^*$.
We view $\cM^*$ as an $\H$-module with respect to the action defined by
\[ (H L)(m) = L(H^\dag m)\qquad\text{for $H \in \H$ and $L \in \cM^*$ and $m \in \cM$,}\]
where  $H\mapsto H^\dag$ denotes the $\cA$-algebra anti-automorphism of $\H$ with 
$(H_w)^\dag=  H_{w^{-1}} $ for $w \in W$.

\begin{theorem}\label{inversion-thm} If $W$ is finite, 
then for all $x,y \in \I_\QP$ it holds that
\[\sum_{w \in \I_\QP} (-1)^{\frac{\ell(x)+\ell(w)}{2}} \cdot  m_{x,w}\cdot  n_{yw_0^+,ww_0^+} 
=\delta_{x,y}.
\]
\end{theorem}

\begin{remark} 
Recall that $m_{x,y} = n_{x,y}=0$ unless $x$ and $y$ belong to the same $W$-conjugacy class, in which case $\ell(y)-\ell(x)$ is even, so the exponentiation of $-1$ in this formula is well-defined.
\end{remark}

\begin{remark}
An analogous inversion formula, due to Douglass \cite{Dou}, exists for the polynomials $m_{x,y}$ and $n_{x,y}$ defined relative to the quasiparabolic set $(W^J,\ell)$ when $W$ is finite (see Example \ref{parabolic-ex}); see  \cite[Proposition 3.9]{Soergel} for a restatement of this formula in notation closer to ours.
\end{remark}

\begin{proof}
Let $\Upsilon : \cM \to \cM^*$ be the $\cA$-linear map with
$ \Upsilon\(\overline {M_{ww_0^+}} \) = M_w^*$ for $w \in \I_\QP$.
Lemmas \ref{barop-lem1} and \ref{w0-lem} ensure that this map is a well-defined $\cA$-linear bijection.
Using the fact that $w \mapsto ww_0^+$ is an involution of $\I_\QP$ which commutes with $W$-conjugation and which reverses the Bruhat order, it is straightforward to check that $\Upsilon$ is moreover an isomorphism of $\H$-modules.
Next, denote by  $L \mapsto \overline{L}$ the $\cA$-antilinear map $\cM^*\to \cM^*$ 
with
\[ \overline{L}(m) = \overline{L(\overline{m})}\qquad\text{for $L \in \cM^*$ and $m \in \cM$}.\]
It follows by Lemma \ref{barop-lem1} that $\overline{ M_w^*} = M_w^*$ if $w \in \I_\QP$ is $W$-maximal, and since $\overline{H}^\dag = \overline{H^\dag}$ for all $H \in \H$, one easily checks that  that $ \overline{H L}(m)  = \(\overline{H}\cdot \overline{L}\)(m) $
for all $H \in \H$ and $L \in \cM^*$ and $m \in \cM$.  
From these properties and the fact that $w \mapsto ww_0^+$ is Bruhat order-reversing on $\I_\QP$,  it follows that map
$M \mapsto \Upsilon^{-1}\(\overline{\Upsilon(M)}\)$ is a bar operator on $\cM$. Since the bar operator on $\cM$ is unique by Proposition \ref{barop-prop1}, it must hold  that
$\overline{M} = \Upsilon^{-1}\(\overline{\Upsilon(M)}\)$ or, equivalently, 
that
\be\label{upsilon} \Upsilon(\overline{M}) = \overline{\Upsilon(M)}\qquad\text{for all $M \in \cM$.}\ee
Now recall the definition of the element $\underline M'_{x} \in \cM$ for $x \in \I_\QP$ from \eqref{qpcanon2-eq}.
Since $\ell(xw_0^+)-\ell(ww_0^+) = \ell(w)-\ell(x)$ and since
$\underline M'_{xw_0^+} = \overline{\underline M'_{xw_0^+}}$, it holds that 
$ \Upsilon\({\underline M'_{xw_0^+}}\) =
 \Upsilon\(\overline{\underline M'_{xw_0^+}}\)$
 %
 %
which means that 
\[ \Upsilon\({\underline M'_{xw_0^+}}\)(\underline M_{y}) = \sum_{w \in \I_\QP} (-1)^{\frac{\ell(x)+\ell(w)}{2}} \cdot n_{ww_0^+,xw_0^+} \cdot m_{w,y}.\]
Since $n_{x,y}$ and $m_{x,y}$ each belong to the set $\delta_{x,y} + v^{-1} \ZZ[v^{-1}]$,
it follows that 
$ \Upsilon\({\underline M'_{xw_0^+}}\)(\underline M_{y}) \in \delta_{x,y} + v^{-1} \ZZ[v^{-1}].$
On the other hand, $ \Upsilon\({\underline M'_{xw_0^+}}\)(\underline M_{y})$ must be invariant under the bar operator on $\cA$ since \eqref{upsilon} combined with the bar invariance of the elements $\underline M_x$ and $\underline M'_x$ implies that
\[\overline{ \Upsilon\({\underline M'_{xw_0^+}}\)(\underline M_{y})}  
= \overline{\Upsilon\({\underline M'_{xw_0^+}}\)}(\overline{\underline M_{y}}) 
= \Upsilon\(\overline{\underline M'_{xw_0^+}}\)(\underline M_{y}) = \Upsilon\({\underline M'_{xw_0^+}}\)(\underline M_{y}).
\]
The only way to reconcile these observations is to conclude that 
\[
\sum_{w \in \I_\QP} (-1)^{\frac{\ell(x)+\ell(w)}{2}} \cdot n_{ww_0^+,xw_0^+} \cdot m_{w,y}=  \Upsilon\({\underline M'_{xw_0^+}}\)(\underline M_{y})= \delta_{x,y}.\]
This identity is equivalent to the statement of the theorem: the theorem asserts that a matrix identity of the form $AB=1$ holds for two certain square matrices $A$ and $B$ whose rows and columns are indexed by $\I_\QP$, and   the preceding identity is the transpose of that equation.
\end{proof}

\begin{corollary} If $W$ is finite, then 
\[ M_x = \sum_{w \in \I_\QP} (-1)^{\frac{\ell(x)+\ell(w)}{2}} \cdot n_{xw_0^+,ww_0^+}\cdot \underline M_w
\quand
N_x = \sum_{w \in \I_\QP} (-1)^{\frac{\ell(x)+\ell(w)}{2}} \cdot m_{xw_0^+,ww_0^+}\cdot \underline N_w
\]
for all $x \in \I_\QP$,
where  $\{M_x\}$ and $\{\underline M_x\}$ (respectively, $\{N_x\}$ and $\{\underline N_x\}$) denote the standard and canonical bases of the $\H$-module $\cM(\I_\QP,\tfrac{1}{2}\ell)$ (respectively, $\cN(\I_\QP,\tfrac{1}{2}\ell)$).
\end{corollary}

\begin{proof}
Expand the canonical basis elements on the right as $\underline M_w = \sum_{y \in \I_\QP} m_{y,w} M_y$ and $\underline N_w = \sum_{y \in \I_\QP} n_{y,w}N_y$, interchange the order of summation, and then apply Theorem \ref{inversion-thm}.
\end{proof}

\section{Problems and conjectures}\label{problem-sect}

We mention some conjectures and problems related to our results.
Recall 
the definition of the notation $\cR(x)$ from \eqref{cR(x)}.
As we noted in Remark \ref{freebar-rmk},
it appears that the only bounded quasiparabolic sets which automatically admit bar operators are those  arising from the parabolic case, in the sense of the following conjecture:

\begin{conjecture} \label{parabolic-conj}
If $(X,\h)$ is a quasiparabolic $W$-set which is transitive and bounded below, and if $|\cR(x)| =1$ for all $x \in X$, then $(X,\h) \cong (W^J,\ell)$ for some $J\subset S$.
\end{conjecture}

\cite[Theorem 3.11.4]{BGS96} summarizes a number of 
 interpretations of the ``parabolic Kazhdan-Lusztig bases'' of $\cM(W^J,\ell)$ and $\cN(W^J,\ell)$  in a representation theoretic context.
Such interpretations lead to the following problem, which is related to the discussion in \cite[\S9]{RV} and \cite[\S10]{RS}.

\begin{problem}\label{geom-prob}
Find a geometric or representation-theoretic interpretation of the quasiparabolic conjugacy classes in $W^+$, of the corresponding modules $\cM$ and $\cN$, and their canonical bases. 
\end{problem}

\begin{remark}
As mentioned by an anonymous referee, there are a few case of quasiparabolic conjugacy classes for which the corresponding $\H$-modules and their bases  do admit natural representation-theoretic interpretations.  These are the conjugacy classes of the fixed-point-free involutions in $S_{2n}$, of the longest element of type $D_4$ in the Weyl group of type $E_6$, and of a perfect involution with maximal proper centralizer in type $D_n$.  These classes correspond to the real semisimple Lie groups $\mathrm{SU}^*(2n)$, ${E_6}^{-26}$, and $\mathrm{SO}(1,2n-1)$, and the associated $\H$-module basis elements  correspond to standard representations of these Lie groups whose central characters are the same as that of the trivial representation; see \cite{LV0,Vogan}.
\end{remark}

The following conjecture   is stated implicitly in \cite[\S5]{RV}, and proved in the special case of  $W$-conjugacy classes of  automorphisms $\theta \in \Aut(W,S)\subset W^+$ which are perfect involutions \cite[Proposition 5.17]{RV}. This conjecture seems to closely parallel the main result of \cite{RS}.

\begin{conjecture}\label{bruhat-conj} The ``Bruhat order''  on a quasiparabolic $W$-conjugacy class in $W^+$ as given by Definition \ref{bruhat-def} coincides with the restriction of the usual Bruhat order on $W^+$.
\end{conjecture}

As Rains and Vazirani note in \cite{RV}, the  criterion that any perfect conjugacy class of twisted involutions is quasiparabolic  is often inadequate in applications involving infinite Coxeter groups. 

\begin{problem}
Formulate a  version of Theorem \ref{perfect-thm} which can be used to prove that (interesting) conjugacy classes in $W^+$ are quasiparabolic when $W$ is infinite. 
Classify the quasiparabolic conjugacy classes in $W^+$ when $(W,S)$ is an affine Weyl group.
\end{problem}

It appears that quasiparabolic $W$-conjugacy classes in $\I$ may be characterized by a simpler set of conditions than the ones in Definition \ref{qp-def}. 
Specifically, we conjecture the following:

\begin{conjecture} 
Any conjugacy class in $\I$ which satisfies  property (QP1) in Definition \ref{qp-def} (relative to the height function $\h = \frac{1}{2}\ell$) also satisfies (QP2), and hence is quasiparabolic.
\end{conjecture}

A lot of useful technical machinery has been developed for twisted involutions in a Coxeter group; see, for example, \cite{H1,H2,H3,R,RS,S}. One reason to expect the preceding conjecture to be true is that it reduces via this machinery to the following second conjecture, which can be viewed as a plausible ``strong exchange condition'' for twisted involutions, analogous to Hultman's ``(weak) exchange condition'' \cite[Proposition 3.10]{H2}.
Recall here that $R = \{ wsw^{-1} : (w,s) \in W\times S\}$.

\begin{conjecture}
Let $\cK\subset \I$ be a $W$-conjugacy class such that $\ell(rwr) = \ell(w)$ implies $rwr=w$ for all $ (r,w) \in R \times \cK$. 
Then $\ell(rwr) < \ell(w)$ implies $rwr<w$ for all $(r,w) \in R \times \cK$.
\end{conjecture}

Our results in Section \ref{Wgraph-sect}   lead to the following problem.

\begin{problem} Describe the cells of the $W$-graphs $\Gamma_m$ and $\Gamma_n$ attached via Theorem \ref{Wgraph-thm} to a quasiparabolic conjugacy class 
in a finite or affine Weyl group. 
\end{problem}

 In the classical cases, this problem is of interest just in view of the elegant combinatorial description of the left cells in the symmetric group (see \cite[Chapter 6]{CCG}). More generally, it would be especially interesting to connect information about the cells in $\Gamma_m$ and $\Gamma_n$ to  Problem \ref{geom-prob}.

As discussed in Remark \ref{lvremark}, Lusztig and Vogan  \cite{LV2,LV1,LV3} have recently studied  an Iwahori-Hecke algebra module spanned by the entire set of twisted involutions $\I$ in a Coxeter group, which admits a  ``bar operator''  given formally by nearly the same definition as for the bar operators in Theorem \ref{Ibarop-thm}. 
Despite this, it remains unclear whether the  canonical bases corresponding to these bar operators have any simple relationship. 

\begin{problem}\label{lastprob} How are the bases $\{ \underline M_x\}_{x \in \I_\QP}$ and $\{ \underline N_x\}_{x \in \I_\QP}$ defined by Theorems \ref{qpcanon-thm} and \ref{Ibarop-thm} related to the canonical basis indexed by $\I$ studied in \cite{LV2,LV1,LV3}?
\end{problem}

\end{document}